\documentclass[12pt]{article}

\usepackage{amsmath,amssymb,amsthm,url}

\usepackage{dsfont}
\usepackage{multirow}
\usepackage{booktabs}
\usepackage{placeins}
\usepackage[a4paper,left=2.5cm,right=2.5cm,top=2.5cm,bottom=2.5cm]{geometry}

\newtheorem{Theorem}{Theorem}
\newtheorem{lemm}{Lemma}
\newtheorem{Assumption}{Assumption}
\newtheorem{Proposition}{Proposition}
\newtheorem{Remark}{Remark}
\newtheorem{Corollary}{Corollary}
\newtheorem{Example}{Example}

\usepackage{natbib}

\usepackage{graphicx}
\usepackage{hyperref}

\usepackage[shortlabels]{enumitem}
\usepackage[utf8]{inputenc}
\usepackage[T1]{fontenc}

\DeclareMathOperator*{\var}{Var}

\newcommand{\cs}{\mathcal{S}}
\newcommand{\invtn}{{\dagger}}
\newcommand{\convP}{\stackrel{P}{\to}}
\newcommand{\convd}{\stackrel{d}{\to}}

\title{Estimating the conditional distribution in functional regression problems}

\usepackage{authblk}

\author[a]{Siegfried H\"ormann}
\author[a]{Thomas Kuenzer}
\author[b]{Gregory Rice}
\affil[a]{Institute of Statistics, Graz University of Technology, Graz, Austria}
\affil[b]{Department of Statistics and Actuarial Science, University of Waterloo, Canada}

\begin{document}

\maketitle

\begin{abstract}
We consider the problem of consistently estimating the conditional distribution $P(Y \in A |X)$ of a functional data object $Y=(Y(t): t\in[0,1])$ given covariates $X$ in a general space, assuming that $Y$ and $X$ are related by a functional linear regression model.
Two natural estimation methods are proposed, based on either bootstrapping the estimated model residuals, or fitting functional parametric models to the model residuals and estimating $P(Y \in A |X)$ via simulation. Whether either of these methods lead to consistent estimation depends on the consistency properties of the regression operator estimator, and the space within which $Y$ is viewed.
We show that under general consistency conditions on the regression operator estimator, which hold for certain functional principal component based estimators, consistent estimation of the conditional distribution can be achieved, both when $Y$ is an element of a separable Hilbert space, and when $Y$ is an element of the Banach space of continuous functions.
The latter results imply that sets $A$ that specify path properties of $Y$, which are of interest in applications, can be considered. The proposed methods are studied in several simulation experiments, and data analyses of electricity price and pollution curves.
\end{abstract}

\bigskip


\clearpage

\section{Introduction}

We suppose that we have observed data $(Y_1,X_1),\ldots, (Y_n,X_n)$ from a strictly stationary process $(Y_k,X_k)_{k\in \mathbb{Z}}$  that are assumed to follow  a general  functional linear regression model of the form
\begin{equation}\label{mod}
Y_k=\varrho(X_k)+\varepsilon_k.
\end{equation}
Here $Y_k=(Y_k(t)\colon t\in [0,1])$ is a curve in a normed function space $H_2$, the covariates $X_k$ take values in a normed space $H_1$ and are distributed so that $X_k$ is independent of the model error $\varepsilon_k$, and $\varrho$ is a linear operator mapping $H_1$ to $H_2$. For example, $X_k$ might be a single curve living in the same space as the response, in which case \eqref{mod} describes simple linear function-on-function regression. This setting also includes functional autoregressive models \citep{bosq:2000} when $X_k = Y_{k-1}$. Generally though, $X_k$ might be comprised of several curves, a mixture of curves and scalar covariates, etc., and more detailed assumptions on the nature of the space $H_2$ will follow.

Suppose $(Y,X)$ is a generic pair following \eqref{mod}. The goal of this paper is to introduce and study methods to consistently estimate the conditional distribution of $Y$ given $X$, $P(Y\in A|X)$, for some specific sets of interest $A\subset H_2$. By choosing appropriate sets $A$, one may make inference on a wide range of interesting properties of $Y$:
\begin{enumerate}[topsep=5pt,itemsep=2pt]
\item  Often we are interested in some transformation $T$ of the response, and then might consider sets of the form
$A=\{y\colon T(y)\in B\}$. For instance, %
when $Ty=\lambda(\{t : y(t) \in B\})$, with $\lambda$ denoting standard Lebesgue measure on $[0,1]$, $A$ contains curves  that occupy a range of interest for a certain amount of time. More generally, when $T(y)$ is a scalar, then we are often interested in the conditional distribution function $$F(z|X)=P(Y\in T^{-1}(-\infty,z]|X).$$

\item Similarly, when $Z=T(Y)$ is again a scalar, for $p\in (0,1)$, we may wish to estimate the conditional quantile function $q_p(Z|X):=\inf\{z\in \mathbb{R}\colon F(z|X)\geq p\}$. In financial applications and when $p$ is close to zero or one, then estimating $q_p(Z|X)$ is related to  Value-at-Risk (VaR) estimation. See \cite{kato:2012} and \cite{sang:2020}.
\item We might wish to choose $A$ such that it yields a prediction set for $Y$, so that $P(Y\in A_p|X)=p$ for a given $p\in (0,1)$. Estimating $P(Y\in A_p|X)$ can be used to appropriately calibrate $A_p$. See \cite{goldsmith:2013}, \cite{choi:2016},  \cite{liebl:2019}, \cite{hyndman:shang:2009}, and \cite{paparoditis:2020} for a review of methods for constructing prediction sets for functional responses and parameters.
\end{enumerate}

At this point, when referring to examples (a) and (b), an important remark is necessary. In the case where $Z=T(Y)$ is scalar, it might appear more natural to directly employ some scalar-on-function regression with response variable $Z$. \emph{However, one of the main strengths of the approach we pursue and which is a clear distinction to competitive methods, is that we first model the entire response curve and then extract the feature of interest.
This has the advantage that we can harness the full information contained in the functional responses when estimating the conditional distribution of $Z$.} 

Aside from interest in the general problem, this work was primarily motivated by the statistical challenge of forecasting aspects of response curves $Y_k$ describing daily electricity prices. The specific data that we consider consists of hourly electricity prices, demand, and wind energy production in Spain over the period from 2014 to 2019, which includes observations from 2191 days (the data are available at \texttt{www.esios.ree.es}). We project the hourly data onto a basis of 18 twice differentiable B-splines to construct daily  price, demand, and wind energy production curves, as illustrated in Figure~\ref{fig:spanishdata}. The price of electricity naturally fluctuates based on supply and demand, and exhibits daily, weekly, and yearly seasonality. The rather predictable variation in demand does not influence
the price as much as surges in wind energy production, especially if they occur on days with weak demand. Letting $Y_k$ denote the price curves and $X_k$ the vector of the demand and wind curves, both adjusted for yearly seasonality and trends, we then model $Y_k$ using an FAR(7) model with exogenous variables
\begin{equation}\label{e:spainfarx}
 Y_k = \sum_{i=1}^7 \Psi_i  Y_{k-i} + \varrho X_k + \varepsilon_k,
\end{equation}
where $\Psi_1,...,\Psi_7$ denote autoregressive operators; see  \cite{gonzalez:munoz:perez:2018}. The details of this are explained in Section~\ref{s:realdata:sim}, but for now it suffices to acknowledge that this is a regression model of the form \eqref{mod}. For such electricity price curves, their likelihood of falling within sets of the following type are of particular interest:

\begin{Example}[Level sets]\label{E:levelsets}
Let
\[
A_{\alpha,z} = \big\{ y\in H_2\colon \lambda(t\colon y(t)>\alpha)\leq z \big\}
\]
for some $z\in [0,1]$ and $\alpha\in \mathbb{R}$. $A_{\alpha,z}$ contains curves that stay a limited amount of time $z$ above a threshold $\alpha$.
\end{Example}

Forecasting whether price or demand curves will spend prolonged periods of time above certain levels is useful in anticipating volatility in continuous intraday electricity markets, and planning for peak loads \citep{vilar:2012}. This falls within the scope of the general problem we consider.

\begin{figure}[t]
\includegraphics[width=\textwidth, trim=0 18 0 2mm,clip]{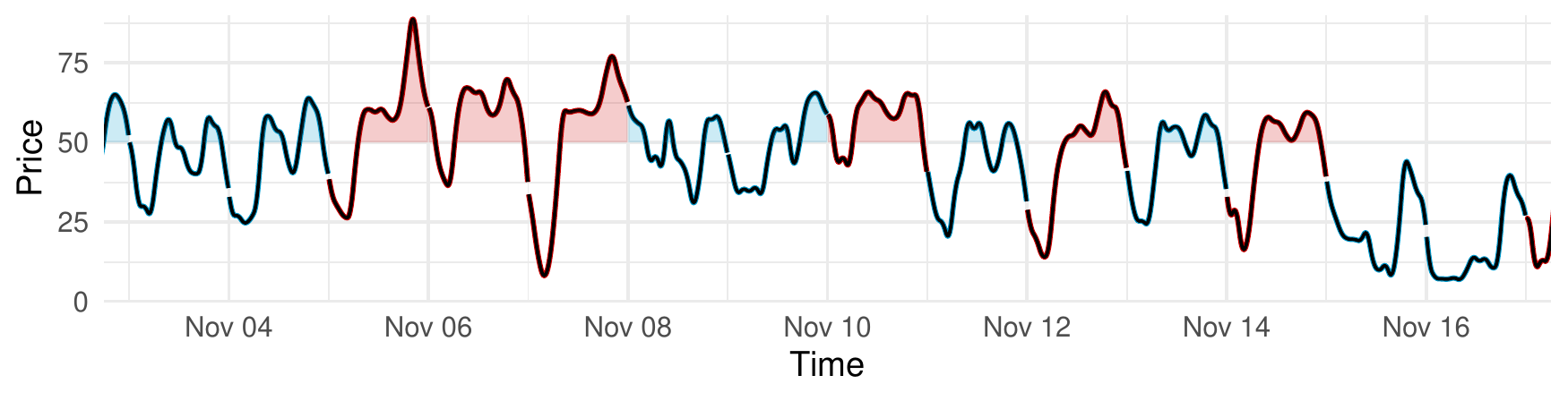}
\includegraphics[width=\textwidth, trim=0 29 0 1mm,clip]{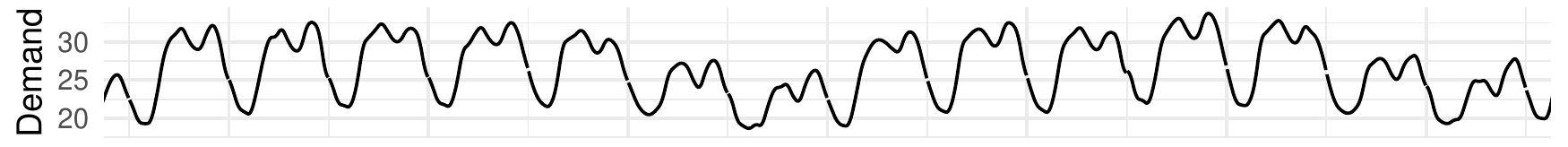}
\includegraphics[width=\textwidth, trim=0 18 0 1mm,clip]{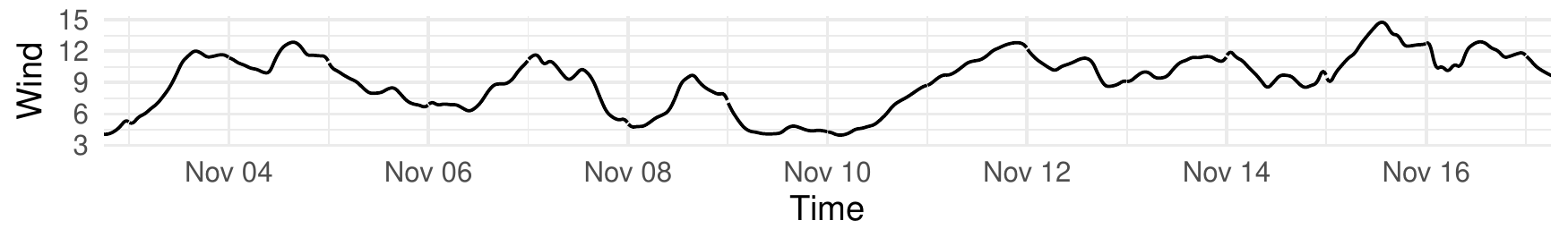}
\includegraphics[width=\textwidth, trim=3 5 0 0mm,clip]{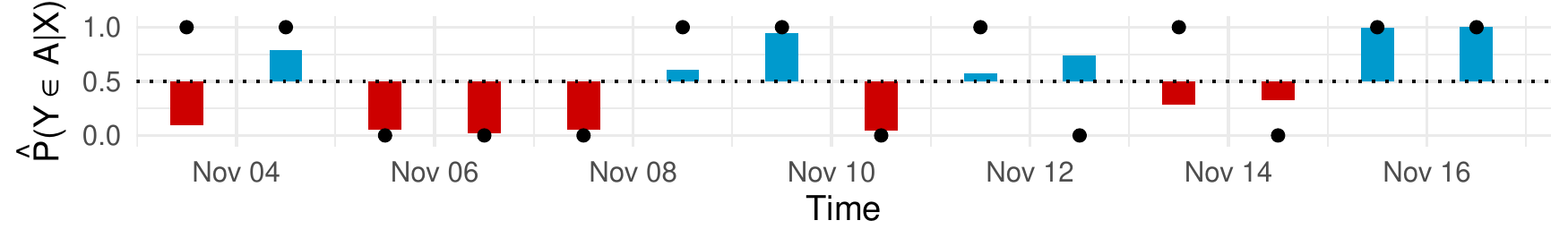}
\caption{Spanish electricity data on price, demand, and wind energy production during two weeks in November 2014. Price curves are colored blue or red according to whether or not
they lie in the level set $\{y \in H_2: \lambda(t: y(t) > 50) \leq 0.5 \}$.
The bar plot on the bottom shows
the estimated conditional probability for $Y_k$ to lie in this set,
with the decision threshold $1/2$, indicated by a dotted line, whether the event occurred is indicated by black dots.}
\label{fig:spanishdata}
\end{figure}

The literature in functional data analysis on regression models of the form \eqref{mod} is vast, although the most frequent problems considered regarding consistent estimation in model \eqref{mod} are (i) how to find a consistent estimator $\hat\varrho_n$  of $\varrho$, and (ii) how to forecast consistently, i.e.\ to guarantee that $\hat\varrho_n(X)-\varrho(X)\to 0$ suitably in probability. Moreover, the majority of the literature on the topic of  function-on-function linear regression concentrates on the  setting when $H_1=H_2=L^2[0,1]$, the separable Hilbert space of square integrable functions on $[0,1]$, equipped with its standard inner product and norm.  \cite{ramsay:silverman:2006} for example proposes a double truncation scheme based on functional principal component analysis to estimate $\varrho$ in this setting, and \cite{mas:2007}, \cite{imaizumi:kato:2018} derive a convergence rate for
$\| \hat\varrho_n - \varrho \|_\cs$ in a ``single-truncation'' estimation scheme based on an increasing (in the sample size) number of principal components, where here $\| \cdot \|_\cs$ denotes the Hilbert--Schmidt norm. Similar consistency results for the resulting forecasts in functional linear regression can be found in \cite{crambes:mas:2013}, and under general stationarity conditions and in the FAR setting in \cite{hormann:kidzinski:2015} and \cite{aue:2015}. Estimating the operator $\varrho$ can be viewed as a special case of estimating the conditional mean $E(Y|X)$, and this general problem has also been extensively considered; see \cite{Chiou04functionalresponse}, \cite{ferraty:2012:fonfregressionbootstrap}, and \cite{wang:chiou:muller:review}.

The problem of estimating the conditional distribution of $Y$ given $X$ has been comparatively far less studied. Numerous methods have been proposed to estimate the conditional distribution of a scalar response $Y$ with a functional covariate $X$, including  \cite{chen:muller:2012:funquantile}, \cite{kato:2012}, \cite{yao:suechee:wang}, \cite{wang:chiou:muller:review}, and \cite{sang:2020}, who propose estimators based on quantile regression, and \cite{ferraty:vieu:2006}, who propose Nadaraya--Watson style kernel-smoothed estimators. Estimating the conditional distribution of $Y$ when $Y$ and $X$ take values in a general function space is largely unexplored to our knowledge, even in the context of model \eqref{mod}. \cite{fernandez:guillas:manteiga:2005:bootstrap} and \cite{paparoditis:2020} develop bootstrap procedures based on functional principal component analysis to produce prediction sets in the context of forecasting with Hilbertian FAR models, which can be viewed as a special case of this problem. For functional data taking values in $L^2[0,1]$, \cite{chen:muller:2016} and \cite{fan:muller:2021} develop methods for estimating the conditional distribution of $Y$ given $X$ assuming $X$ and $Y$ are jointly Gaussian, and that the conditional distribution of the response has sample paths satisfying natural differentiability conditions.

A technical problem that is encountered in consistently estimating the conditional probability $P(Y\in A | X)$ is that one could at most expect consistent estimation for continuity sets of the distribution of the response, i.e. sets $A$ for which $P(Y\in \partial A )=0$. This property evidently depends strongly on the choice of the space $H_2$, as well as the norm that it is equipped with. For many interesting examples, the metric on the space $L^2[0,1]$ is too weak to allow for meaningful continuity sets $A$. An illustrative example is simple prediction band sets of the form $A=\{y\colon \lambda( t : a(t)<y(t)<b(t))=1\}$, where $a$ and $b$ are continuous functions on $[0,1]$, in which case $\partial A=A$ when $A$ is viewed as a subset of $L^2[0,1]$. More appropriate spaces to handle many interesting examples in functional data analysis involving path properties of the response, such as level sets, are the spaces $C[0,1]$, the space of continuous function on $[0,1]$ equipped with the supremum norm, or the Sobolev spaces equipped with their canonical norms; see \cite{brezis:2010}. For these latter spaces, the problem of estimating, and consistently forecasting with, $\varrho$ has been only lightly studied to date, and in specialized settings; see \cite{pumo:1999}, \cite{RUIZMEDINA:2019:banachFAR}, and \cite{bosq:2000} in the context of FAR estimation. Further, the problem of consistently estimating the conditional distribution $P(Y\in A | X)$ in these settings has not been studied, to our knowledge. We also refer the reader to \cite{dette:kokot:aue:2020} for a review of functional data analysis methods in $C[0,1]$.

In this paper, we propose natural procedures to estimate $P(Y\in A |X)$, in which we first estimate $\varrho$ with a suitably consistent estimator $\hat{\varrho}_n$, and then either (i) resample the estimated residuals $\hat{\varepsilon}_{k,n} = Y_k  - \hat{\varrho}_n(X_k)$ to estimate $P(Y\in A |X)$ with the empirical distribution of $\hat{\varrho}_n(X) + \hat{\varepsilon}_{k,n}$, or (ii) assuming Gaussianity of the model errors $\varepsilon_k$, we estimate $P(Y\in A |X)$ using simulation by modelling $Y$ conditioned on $X$ as a Gaussian process with mean $\hat{\varrho}_n(X)$, and covariance estimated from the residual sequence $\hat{\varepsilon}_{k,n}$. We establish general conditions on the estimator $\hat{\varrho}_n$ in both the settings when $H_2$ is a separable Hilbert space, and when $H_2$ is $C[0,1]$, such that these procedures will lead to consistent estimation of $P(Y\in A |X)$.
Subsequent to this, we define an estimator $\hat{\varrho}_n$ that we show satisfies these conditions under regularity assumptions on the operator $\varrho$, and the process $(Y_k,X_k)_{k\in \mathbb{Z}}$, which allow for serial dependence of both the response and covariates. In the space $H_2=C[0,1]$, we introduce a number of examples of sets $A$ of potential interest, compute their boundaries, and establish under what conditions $P(Y\in \partial A )=0$, which can be non-trivial even when $Y$ is a Gaussian process. In several simulation studies and data analyses, the proposed methods generally outperformed competing methods in which $P(Y\in A |X)$ is estimated using functional logistic regression, functional Nadaraya--Watson estimation, or functional quantile regression. Another advantage, which derives from the simple form of our estimators for $P(Y\in A |X)$, is that they satisfy basic properties of a probability measure, e.g.\ they are monotone in $A$. While this may seem like an obvious requirement, it is not necessarily fulfilled by some competing approaches.

The rest of the paper is organized as follows. In Section~\ref{sec-1}, we formally introduce the methods described above to estimate $P(Y\in A |X)$, and present results on their consistency, including results on uniform consistency over monotone families of sets $A$ that are relevant in constructing prediction sets with a specified coverage and quantile function estimates. These results depend on the properties of the estimator $\hat{\varrho}_n$, and we define an estimator based on functional principal component analysis and single truncation scheme, and establish that it leads to consistent estimation of $P(Y\in A |X)$ when $H_2$ is a separable Hilbert space and when $H_2=C[0,1]$ in Section~\ref{s:theoretical}. Section~\ref{s:partialA} presents numerous examples of sets $A$ of interest, and a discussion of their boundaries in $C[0,1]$. A number of competing methods are introduced in Section~\ref{s:realdata:sim}, and these are compared and studied with the proposed methods in several simulations studies and real data illustrations. The proofs of all technical results follow these main sections.

\section{Estimation procedures and consistency results}\label{sec-1}

We let $\int = \int_{0}^{1}$, and for $f,g \in L^2[0,1]$, we use the notation $\langle f , g \rangle = \int f(t)g(t)dt$ to denote the standard inner product on $L^2[0,1]$, with induced norm $\|\cdot\|_{L^2}^2  = \langle \cdot , \cdot \rangle$.
In order to consider path properties of functions in $H_2$, we consider the space $C[0,1]$ equipped with the supremum norm $\|f\|_\infty = \sup_{t\in [0,1]} \big| f(t) \big|$. While $L^2[0,1]$ is a separable Hilbert space, $C[0,1]$ is a Banach space, with their norms satisfying $\|\cdot\|_{L^2} \leq \|\cdot\|_{\infty}$. The space $C[0,1]$ may hence be naturally embedded in $L^2[0,1]$.
We use the tensor product notation $\otimes$ to denote the operator $a\otimes b(\cdot) = a\langle \cdot,b\rangle$ if $b$ is viewed as an element of a Hilbert space, and the kernel integral operator with kernel $a\otimes b (t,s) = a(t)b(s)$ if $b$ is viewed as an element of $C[0,1]$. We assume that the covariate space $H_1$ is a Hilbert space with norm
$\|\cdot\|_{H_1}$. In order to lighten the notation, and when it is clear from the context, we write $\|\cdot\|$ in place of a specific norm on either the space $H_1$ or $H_2$. In the context of Hilbert spaces,
we use $\|\cdot\|_\mathcal{S}$ and $\|\cdot\|_1$ to denote the Hilbert--Schmidt norm and the trace norms of operators, respectively.

We assume throughout that the covariates $X_k$ and the model errors $\varepsilon_k$ satisfy the following independence condition, which we do not explicitly state in the below results, but take as granted.

\begin{Assumption}\label{a:errorind} In model \eqref{mod}, $E\varepsilon_k = 0$, and $\varepsilon_k$  is independent from $(X_j)_{j\leq k}$ for all $k \in \mathbb{Z}$.
\end{Assumption}

In order to formally describe the methods we use to estimate $P(Y\in A|X)$, we assume that we may consistently estimate $\varrho$ with an estimator $\hat{\varrho}_n$ based on the sample  $(Y_1,X_1),...,(Y_n,X_n)$. Specific conditions on this estimator, and examples satisfying these conditions, will follow. The first method we describe is based on applying an i.i.d.\ bootstrap to the estimated residuals.

\medskip

\noindent\underline{ Algorithm~1, Residual Bootstrap (abbreviated {\bf boot}):}
\begin{enumerate}
\item Estimate $\varrho$ in \eqref{mod} with $\hat\varrho_n$.
\item Calculate the model residuals $\hat\varepsilon_{k,n} = Y_k-\hat\varrho_n(X_k)$.
\item Define the estimator of $P(Y\in A|X)$ as
\[
\hat{P}_n^\text{B}(Y\in A|X)=\frac{1}{n}\sum_{k=1}^n \mathds{1}\{\hat\varrho_n(X)+\hat\varepsilon_{k,n}\in A\}.
\]
\end{enumerate}

We show below that this estimator is consistent under quite mild conditions. The algorithm {\bf boot} can be applied without specific distributional assumptions on the errors.
If though the model errors $\varepsilon_k$ are thought to be Gaussian processes, then we may take this into account in estimating $P(Y\in A |X)$. As the mean of the model errors is zero by assumption, in this case their distribution is determined by their covariance $\Gamma = \var(\varepsilon_k)$, which is defined by
\[
\var(\varepsilon_k) = E\big[ \varepsilon_k \otimes \varepsilon_k \big].
\]
The above algorithm may then be adapted as follows:

\medskip

\noindent\underline{ Algorithm~2, Gaussian process simulation (abbreviated {\bf Gauss}):}
\begin{enumerate}
\item Estimate $\varrho$ in \eqref{mod} with $\hat\varrho_n$.
\item Calculate the model residuals $\hat\varepsilon_{k,n} = Y_k-\hat\varrho_n(X_k)$.
\item Estimate the empirical covariance operator
\begin{align}
\widehat\Gamma_{\varepsilon,n} = \frac{1}{n}\sum_{k=1}^n (\hat\varepsilon_{k,n}-\bar{\varepsilon}_{\cdot,n})\otimes(\hat\varepsilon_{k,n}-\bar{\varepsilon}_{\cdot,n}),
\label{e:gammahat}
\end{align}
where $\bar{\varepsilon}_{\cdot,n}=\frac{1}{n}\sum_{k=1}^n \hat\varepsilon_{k,n}$. Let $\hat{\nu}_1 \ge \hat{\nu}_2 \ge \cdots $ denote the ordered eigenvalues of $\widehat\Gamma_{\varepsilon,n}$, with corresponding eigenfunctions $(\hat{\psi}_j)_{j\ge 1}$ satisfying $\widehat\Gamma_{\varepsilon,n}(\hat{\psi}_j) = \hat{\nu}_j \hat{\psi}_j$, $\langle \hat{\psi}_j, \hat{\psi}_\ell \rangle = \mathds{1}\{j = \ell\}$.

\item  Let $\{Z_i, \; i\ge 1\}$ denote a sequence of i.i.d.\ standard normal random variables, independent of the sample $(Y_1,X_1),...,(Y_n,X_n)$, and define
\[
\varepsilon^{(n)} = \sum_{j=1}^{\infty}\hat{\nu}_j^{1/2}  Z_j \hat{\psi}_j.
\]
Note that conditionally on the sample, in particular on $\widehat\Gamma_{\varepsilon,n}$, $\varepsilon^{(n)}$ is a Gaussian process with mean zero and covariance operator $\widehat\Gamma_{\varepsilon,n}$. Define the estimator of $P(Y\in A|X)$ as
\begin{align}\label{eq:gprob}
\hat{P}^\text{G}_{n}(Y\in A|X) =P(\hat\varrho_n(X)+\varepsilon^{(n)}\in A|X).
\end{align}

The right hand side above can be approximated by Monte-Carlo simulation. To do so, generate an i.i.d.\ sample  conditionally on $\widehat\Gamma_{\varepsilon,n}$,  $(\varepsilon_k^{(n)})_{k\geq 1}$, distributed as $\varepsilon^{(n)}$, by simulating independent standard Gaussian sequences $\{Z_{i,k}, \; i\ge 1\}_{k\ge1}$  and setting
\[
\varepsilon^{(n)}_k = \sum_{j=1}^{\infty}\hat{\nu}_j^{1/2}  Z_{j,k} \hat{\psi}_j.
\]
The right hand side of \eqref{eq:gprob} can be estimated, for a large $M$, by
\[
\frac{1}{M}\sum_{k=1}^M \mathds{1}\{\hat\varrho_n(X)+\varepsilon_k^{(n)}\in A\}.
\]

\end{enumerate}

\begin{Remark} {\rm  The scaling $1/n$ in the definition of $\widehat\Gamma_{\varepsilon,n}$
does not take into account the degrees of freedom $T_n$ lost in the estimation of the regression operator $\varrho$. It has thus been advocated, for example in \cite{crambes:2016},  to instead divide by $n - T_n$, where $T_n$ is related to the dimension of the dimensionality reduction technique used in estimating $\hat{\varrho}_n$. If $ET_n = o(n)$, as is the case for most estimation approaches,
the resulting scaling difference is asymptotically negligible.
Some authors also propose splitting the sample
and estimating the regression operator and the noise covariance operator
on separate parts of the sample in order to reduce the bias of the estimator $\widehat\Gamma_{\varepsilon,n}$; see \cite{crambes:mas:2013}. }
\end{Remark}

\begin{Remark} {\rm
Algorithm {\bf Gauss} can be extended to other parametric distributions of the noise.
A notable example for this are infinite dimensional elliptic distributions, where
$\varepsilon_k = \Xi_k   \varepsilon^\prime_k$
with two independent random variables
$\varepsilon^\prime_k \in H_2$, which is Gaussian, and $\Xi_k \geq 0$ from a known univariate parametric distribution. The following investigation can easily be adapted to this setting. For details on elliptical distributions of functional data, we refer to \cite{boente:2014}. }
\end{Remark}

We now aim to establish consistency results for these algorithms. In order to keep the results as general as possible and allow for different estimators of $\varrho$, these results are stated in terms of the following consistency properties of $\hat{\varrho}_n$.

\begin{Assumption}\label{a:consist}
The estimator $\hat\varrho_n$ is such that
\begin{enumerate}[(a),topsep=0pt,itemsep=0pt]
\item \label{a:consistOut}
its out-of-sample prediction is consistent,
i.e.\ if $X \stackrel{d}{=} X_1$, and is independent from the sample, then
\[
\|\hat\varrho_n(X)-\varrho(X)\| \convP 0,
\quad  n \to \infty.
\]
\item \label{a:consistIn}
its in-sample prediction is consistent,
i.e.\ let $K_n$ be independent from the sample and
uniformly distributed on $\{1,\dots,n\}$, then
\[
\|\hat\varrho_n(X_{K_n})-\varrho(X_{K_n})\| \convP 0,
\quad  n \to \infty.
\]
\end{enumerate}
\end{Assumption}

In order to establish the consistency of the algorithm {\bf Gauss}, we additionally need conditions on the estimator of the covariance
operator of the model errors defined by \eqref{e:gammahat}. We state two conditions depending on whether $H_2$ is a separable Hilbert space, or $H_2 = C[0,1]$.

\begin{Assumption}\label{a:consist2}
 $H_2$ is a separable Hilbert space.
The estimator $\hat\varrho_n$ is such that
\[
E \Bigg\| \frac 1n \sum_{i=1}^n (\hat\varrho_n - \varrho) X_k \otimes X_k (\hat\varrho_n - \varrho)^* \Bigg\|_1 \to 0, \mbox{ as }
\quad n\to\infty.
\]

\end{Assumption}

\begin{Assumption}\label{a:consist2C}
~
\begin{enumerate}[(a),topsep=0pt,itemsep=0pt]
\item $H_2 = C[0,1]$. The estimator $\hat\varrho_n$ is such that
\[
\sup_{t,s\in[0,1]} \Bigg| \frac 1n \sum_{i=1}^n
(\hat \varrho_n - \varrho)(X_k)\otimes (\hat \varrho_n - \varrho)(X_k)(t,s)
\Bigg| \convP 0,
\quad n \to \infty.
\]
\item The estimated variance of the model errors,
\[
V_n^2(t,s) = \var\big( \varepsilon^{(n)}(t) - \varepsilon^{(n)}(s) \big| \widehat\Gamma_{\varepsilon,n} \big)
 = \widehat\Gamma_{\varepsilon,n}(t,t) - 2 \widehat\Gamma_{\varepsilon,n}(t,s) + \widehat\Gamma_{\varepsilon,n}(s,s)
\]
satisfies the H\"{o}lder condition
\[
V_n^2(t,s) < M_V^2 \, |t-s|^{2\alpha}, \quad    t,s\in[0,1],
\]
for some $0<\alpha\leq 1$,
where $M_V$ is a positive random variable with
$E M_V < \infty$.
\end{enumerate}
\end{Assumption}

Assumption~\ref{a:consist2C} is a $C[0,1]$ analog of Assumption~\ref{a:consist2}, but with the addition of Assumption~\ref{a:consist2C}~(b) that implicitly demands a degree of  continuity of $\varrho(X_k)$ and the model errors $\varepsilon_k$. Under the above assumptions, we can now formulate our main consistency results.

\begin{Theorem}\label{t:main1}
Suppose that Assumption~\ref{a:consist} holds
and that $P(Y\in \partial A)=0$.
Then
$\hat P^\text{B}_n(Y\in A|X)\convP P(Y\in A|X)$ as $n\to\infty$.
\end{Theorem}

\begin{Theorem}\label{t:main2}
Suppose that Assumption~\ref{a:consist}~\ref{a:consistOut} holds and
either Assumption~\ref{a:consist2} or Assumption~\ref{a:consist2C} holds.
Assume that
$(\varepsilon_k)_{k\geq 1}$ are i.i.d.\ Gaussian random variables in $H_2$,
and that $P(Y\in \partial A)=0$.
Then
$\hat P^\text{G}_n(Y\in A|X)\convP P(Y\in A|X)$ as $n\to\infty$.
\end{Theorem}

Theorems~\ref{t:main1} and \ref{t:main2} show that consistent estimation of $P(Y\in A |X)$ can be achieved by both {\bf boot} and {\bf Gauss} when $Y$ takes values in either a separable Hilbert space, or $C[0,1]$, under natural consistency conditions on $\hat{\varrho}_n$, and when $A$ is a continuity set of the response $Y$. We note that these results can be readily extended to sets $A$ that, rather than being fixed, are dependent on the predictor $X$, as well as using the estimator $\hat{\varrho}_n$, so long as there is a certain degree of continuity in relating $\{Y\in A\}$ to $\hat{\varrho}_n(X)$. This is often of interest when constructing prediction sets for the response $Y$, as in the following examples in which it is natural to consider $H_2 = C[0,1]$.

\begin{Example}[Pointwise and uniform prediction sets]\label{example:pred} Suppose $a$ and $b$ are positive functions in $C[0,1]$. Given a covariate $X$, let, for $s\in [0,1]$,
\[
\hat{A}_{a,b}^{(n)}(s)= \big\{y\in C[0,1] \colon \hat{\varrho}_n(X)(s)- a(s) \le y(s) \le \hat{\varrho}_n(X)(s)+ b(s)  \big\} \;\; \text{(Point prediction sets)},
\]
and
\[
\hat{U}_{a,b}^{(n)}= \big\{y\in C[0,1] \colon \lambda(t : \hat{\varrho}_n(X)(t)- a(t) \le y(t) \le \hat{\varrho}_n(X)(t)+ b(t) )=1 \big\} \;\; \text{(Uniform prediction sets)}.
\]
These approximate the sets
\[
 A_{a,b}(s)= \big\{y\in C[0,1] \colon  \varrho(X)(s)- a(s) \le y(s) \le  \varrho(X)(s)+ b(s)  \big\},
\]
and
\[
U_{a,b}= \big\{y\in C[0,1] \colon \lambda(t : \varrho(X)(t)- a(t) \le y(t) \le  \varrho(X)(t)+ b(t) )=1 \big\} .
\]
\end{Example}

\begin{Corollary}\label{cor-predsets}
For some $s\in [0,1]$, let $\hat{A}_{a,b}^{(n)}(s)$, $\hat{U}_{a,b}^{(n)}$, $A_{a,b}(s)$, and $U_{a,b}$ be defined in Example~\ref{example:pred}. Suppose that  $P(Y \in \partial A_{a,b}(s))=0$. %
If Assumption~\ref{a:consist} holds, then
\begin{equation}
\label{e:randomsets}
\hat P^\text{B}_n(Y\in \hat{A}_{a,b}^{(n)}(s)|X)\convP P(Y\in A_{a,b}(s)|X),\quad\text{as $n\to\infty$.}
\end{equation}
If Assumptions~\ref{a:consist}~\ref{a:consistOut} and \ref{a:consist2C} hold, then \eqref{e:randomsets} holds with $\hat P^\text{G}_n$ instead of $P^\text{B}_n$. Under $P(Y \in \partial U_{a,b})=0$, the analogue results hold with the sets $\hat{U}_{a,b}^{(n)}$ and $U_{a,b}$.
\end{Corollary}

\subsection{Uniform consistency over monotone families of sets}\label{s:uniform:mono}

For a potentially unbounded  interval $[a,b] \subset \overline{\mathbb{R}}$, we call a family $\mathcal{A}=\{A_\xi\colon \xi\in [a,b]\}$ of measurable subsets of $H_2$ monotone if the sets $A_\xi$ are increasing or decreasing in $\xi$. Suppose that $A_\xi$ is increasing, the decreasing case can be handled similarly, and that we are interested in finding
\[
\xi_p(X)=\inf\{\xi \in [a,b]\colon P(Y \in A_\xi | X ) \geq p\},\quad p\in (0,1).
\]
As an example where this problem is relevant, consider a scalar transformation of the response $Z=T(Y)$, and suppose we wish to estimate the conditional quantile of $Z$ given the covariate $X$
\begin{align*}
q_p(Z|X)&=\inf\{\xi \in [a,b]\colon P(Z\leq \xi | X ) \geq p\}=\inf\{\xi \in [a,b]\colon P\big(Y \in T^{-1}\big([a,\xi]\big) | X \big) \geq p\}.
\end{align*}
The sets $A_\xi:= T^{-1}\big([a,\xi]\big)$ evidently define a monotone family. Consistent scalar-on-function quantile regression can hence be cast as the problem of consistently estimating $\xi_p(X)$ from the sample, which can be done using $\hat{P}_n^\text{B}$ or $\hat{P}_n^\text{G}$. To this end we consider the estimator
\begin{align}
\hat \xi_p^\text{B} (X):= \inf \big\{ \xi \in [a,b]\colon\, \hat P^\text{B}_n(Y \in A_\xi | X ) \geq p \big\}.
\label{e:predictionxi}
\end{align}
We note that based on the definition  of $\hat{P}^\text{B}_n$, $p\mapsto \hat\xi_p^\text{B}(X)$ is a non-decreasing function in $p$. The same holds for $\hat \xi_p^\text{G}(X)$, which is defined using $\hat{P}_n^\text{G}$. While this observation is rather trivial, in  other approaches to scalar-on-function quantile regression one often has to take special
care in order to guarantee monotonicity of estimators of $q_p(Z|X)$, see e.g.\ \cite{kato:2012}.

The goal is now to show that  $\hat\xi_p^\text{B}(X)\convP\xi_p(X)$ and $\hat\xi_p^\text{G}(X)\convP\xi_p(X)$. In order to do so, we need the following
uniform convergence result for the estimated conditional probabilities.
\begin{Proposition}\label{p:uniformconv}
Let $\{ A_\xi\colon \xi \in [a,b] \}$
be a monotone family of sets such that $P(Y \in A_\xi | X)$ is a.s. continuous in $\xi$.
Suppose the estimator $\hat P_n(Y\in A_\xi|X)$ is non-decreasing, right-continuous, and satisfies
$\hat P_n(Y\in A_\xi|X)\convP  P(Y \in A_\xi | X)$ for all $\xi \in [a,b]$.
Then
\[
\sup_{\xi \in [a,b]} \Big| \hat P_n(Y \in A_\xi | X) - P(Y \in A_\xi | X) \Big| \convP 0,
\quad n \to \infty.
\]
\end{Proposition}
We note that  both $\hat P_n^\text{B}(Y\in A_\xi|X)$ and
$\hat P_n^\text{G}(Y\in A_\xi|X)$ satisfy the conditions of Proposition~\ref{p:uniformconv} under the conditions of Theorems~\ref{t:main1} and \ref{t:main2}.

\begin{Corollary}\label{cor:unif}
Define $\hat\xi_p(X)$ as in  \eqref{e:predictionxi} for a some general estimator $\hat P_n(Y \in A_\xi | X)$.
Under the assumptions of Proposition~\ref{p:uniformconv} with increasing sets $A_\xi$, we
have $P(Y\in A_{\hat\xi_p(X)}|X)\convP p$. If $P(Y \in A_\xi | X)$ is strictly increasing in $\xi$,
then $\hat \xi_p(X) \convP \xi_p(X)$.
\end{Corollary}

\section{Estimation of the regression operator}\label{s:theoretical}

In this section we aim to define an estimator $\hat{\varrho}_n$ that satisfies the consistency conditions detailed in Assumptions~\ref{a:consist}, \ref{a:consist2}, and \ref{a:consist2C}. In order to do so, we make the following assumptions on model \eqref{mod}.

\begin{Assumption}\label{a:hilbertsetting}
\begin{enumerate}[(a),topsep=0pt,itemsep=0pt]
\item $H_1$ is a separable Hilbert space.
\item The process $(X_k)_{k \in \mathbb{Z}}$ has mean zero, and is $L^4$-$m$-approximable
in  $H_1$ (see \cite{hormann:kokoszka:2010}).
\item The operator $\varrho\colon H_1\to H_2$ is a bounded linear operator.
\item The sequence $(\varepsilon_k)_{k \in \mathbb{Z}} $ is a mean zero, i.i.d.\ sequence in $H_2$, and satisfies $E\|\varepsilon_k\|^4<\infty$.
\end{enumerate}
\end{Assumption}

Assumption~\ref{a:hilbertsetting}~(b) supposes that $X_k$ is a (strongly) stationary and ergodic sequence with $E\|X_k\|_{H_1}^4<\infty$, and allows the $X_k$ to be weakly serially dependent in a certain sense. \cite{hormann:kokoszka:2010} show that many commonly studied stationary time series in function space, like FAR processes or functional analogs of GARCH processes, are $L^4$-$m$-approximable under suitable moment conditions.

The estimator that we consider is a truncated (functional) principal components based estimator. Let the empirical covariance operator of $X_k$, and the empirical cross-covariance operator between $Y_k$ and $X_k$, be denoted as
\[
\widehat C_{XX} = \frac 1 n \sum_{k=1}^n X_k \otimes X_k, \quad\text{and}\quad  \widehat C_{YX} = \frac 1 n \sum_{k=1}^n Y_k \otimes X_k.
\]

Letting $\langle \cdot , \cdot \rangle_{H_1}$ denote the inner product on $H_1$, we note that $\widehat C_{XX}$ defines a non-negative sequence of eigenvalues $\hat \lambda_i$, and eigenfunctions $\hat v_i$, satisfying $\widehat C_{XX}(\hat{v}_i) = \hat{\lambda}_i \hat{v}_i $, $\langle \hat{v}_i , \hat{v}_j \rangle_{H_1} = \mathds{1}\{i=j\}$. We then define
\begin{equation}\label{e:rhohat}
\hat \varrho_n(x) := \sum_{i=1}^{T_n} \frac{1}{\hat \lambda_i} \, \widehat C_{YX} \, \hat v_i \otimes \hat v_i (x),
\end{equation}

The estimator \eqref{e:rhohat} only truncates the covariance operator of $X$ in order to obtain a feasible approximation to $\widehat C_{XX}^{-1}$, yielding a so-called ``single-truncated'' estimator. The asymptotic properties of these estimated operators have e.g.\ been studied in \cite{mas:2007} and \cite{hormann:kidzinski:2015}. In order to select the truncation parameter $T_n$ in such a way that leads to asymptotic consistency of $\hat{\varrho}_n$, we use the following criterion:
\begin{equation}\label{e:chooseTn}
T_n = \max\big\{j \geq 1 \colon \hat \lambda_j \geq m_n^{-1} \big\},
\quad \text{with }\; m_n \to \infty.
\end{equation}
Here $m_n$ is a tuning parameter, tending to infinity at a rate specified in the results below.

We note that another standard way to select $T_n$ is to use the percentage of variance explained (PVE) approach, which entails taking
\[
T_n = \min\left\{ d : \frac{\sum_{j=1}^{d} \hat{\lambda}_j}{\sum_{j=1}^{\infty} \hat{\lambda}_j } \ge v \right\},
\]
where $v$ is a user specified percentage treated as a tuning parameter.  While the criterion in \eqref{e:chooseTn} is more transparent in terms of describing the asymptotic consistency of $\hat{\varrho}_n$, since it gives a direct description of the decay rate of the sequence of eigenvalues $\hat{\lambda}_j$, in applications the PVE criterion is prevailing, due to its ease of interpretation. By choosing the associated tuning parameters appropriately, the two criteria may be made comparable.

Now we present results which imply Assumptions~\ref{a:consist}--\ref{a:consist2C}, and hence the consistency of the estimators in the Algorithms {\bf boot} and {\bf Gauss}.
\begin{Proposition}\label{p:main1}
Suppose that $H_2$ is a separable Hilbert space, Assumption~\ref{a:hilbertsetting} holds and
we define $\hat\varrho_n$ as in \eqref{e:rhohat} with $m_n = o\big(\sqrt n\big)$.
Then Assumption~\ref{a:consist} holds.
\end{Proposition}

\begin{Proposition}\label{p:main2}
Suppose that $H_2$ is a separable Hilbert space,
Assumption~\ref{a:hilbertsetting} holds
and that the true regression operator
$\varrho$ is Hilbert--Schmidt. If $\hat\varrho_n$ is defined as in \eqref{e:rhohat} with $m_n = o\big(\sqrt n\big)$,
then Assumption~\ref{a:consist2} holds.
\end{Proposition}

In the case when $H_2 = C[0,1]$, we add the following assumption in addition to Assumption~\ref{a:hilbertsetting}, supposing a degree of smoothness to $\varrho(X)$ and $\varepsilon_k$:
\begin{Assumption}\label{a:continuoussetting} If $H_2 = C[0,1]$, and for some $0 < \alpha \leq 1$,
\begin{enumerate}[(a),topsep=0pt,itemsep=0pt]
\item The model errors $\varepsilon_k$ a.s.\ satisfy the
Hölder condition
\begin{align}
\big| \varepsilon_k(t) - \varepsilon_k(s) \big| &< M_k \, |t-s|^\alpha
\label{e:holdercondition}
\end{align}
where $M_k$  is a positive random variable independent from $X_k$, with $E M_k^2 < \infty$.
\item
For all $x \in H_1$,
the regression operator $\varrho$ satisfies
\[
\big| \varrho x(t) - \varrho x(s) \big| \leq M_\varrho \, \|x\| \, |t-s|^\alpha,
\]
where $M_\varrho$ is a finite constant.
\end{enumerate}
\end{Assumption}

Assumption~\ref{a:continuoussetting}~(a) is fulfilled by a wide range of stochastic processes,
most notably the Brownian motion and the fractional Brownian motion.
Since $\varrho$ is linear, Assumption~\ref{a:continuoussetting}~(b) is a natural formulation of
the Hölder condition for the conditional mean of the response.
In particular, this implies that $\varrho$ is a bounded, compact operator.

\begin{Remark}{\rm
Suppose $H_1=L^2[0,1]$, so that model \eqref{mod} describes function-on-function regression. A frequently employed class of operators $\varrho$ in this setting are  kernel integral  operators, defined by a continuous kernel $\rho \in C[0,1]^2$  as
\[
\varrho x(t) = \int \rho(t,u) x(u) du.
\]
If there exists an $a \in H_1$ such that almost everywhere
\[
\big| \rho(t,u) - \rho(s,u) \big| < a(u) \, |t-s|^\alpha,
\]
then one can easily verify that
\begin{align*}
\big| \varrho x(t) - \varrho x(s) \big|
 &\leq \|a\| \, \|x\| \, |t-s|^\alpha,
\end{align*}
and thus
Assumption~\ref{a:continuoussetting}~(b) is fulfilled.}
\end{Remark}

\begin{Proposition}\label{p:mainc1}
Suppose that Assumption~\ref{a:hilbertsetting} and
Assumption~\ref{a:continuoussetting}~(a) hold, and
we define $\hat\varrho_n$ as in \eqref{e:rhohat}
with $m_n = o\big( n^{\alpha/2} \big)$.
Then Assumption~\ref{a:consist} holds.
\end{Proposition}

\begin{Proposition}\label{p:mainc2}
Suppose that
Assumption~\ref{a:hilbertsetting} and
Assumption~\ref{a:continuoussetting} hold, and
we define $\hat\varrho_n$ as in \eqref{e:rhohat}
with $m_n = o\big( n^{\alpha/2} \big)$.
Then Assumption~\ref{a:consist2C} holds.
\end{Proposition}

The proofs of Propositions~\ref{p:main1}--\ref{p:mainc1} are relegated to Section~\ref{s:proofs},
while the proof of Proposition~\ref{p:mainc2} is given in Appendix~\ref{a:additionalproofs}.

We conclude this section with some technical discussion.
We begin by noting that the sequence $m_n$, which controls
how many principal components of $X_k$ are used in forming $\hat{\varrho}_n$, can be of asymptotically higher order if $H_2$ is a Hilbert space compared to the setting when $H_2=C[0,1]$, and $\alpha < 1$.
In the case where the Hölder exponent in Assumption~\ref{a:continuoussetting} is $\alpha=1$,
the responses $Y_k$ are Lipschitz continuous, which implies they are weakly differentiable. As a result one may then take $H_2=W$, a separable Hilbert space, leading back to the rate condition $m_n = o(\sqrt{n})$. The order $n^{\alpha/2}$ is sufficient but not sharp. In fact, for $\alpha < 1/2$, a different proof yields that $m_n = o\big( n^{1/(2+\alpha^{-1})} \big)$ also leads to consistency. In the case of the Brownian motion, $\alpha=1/2$ demands that $m_n=o\big(n^{1/4}\big)$.
This is still of higher order than that suggested to be used by \cite{hormann:kidzinski:2015}
for consistent estimation of the regression operator in Hilbert spaces.

Our second technical remark concerns the choice of $H_1$. Assumption~\ref{a:hilbertsetting}~(a) requires $H_1$ to be a Hilbert space. While typically this is not a restriction, some care needs to be taken in the case of an FAR model. Here, when we study continuous functions, we choose $H_2=C[0,1]$. While it is natural to assume that the covariate and response space coincide for an FAR (i.e.\ requiring $H_1=C[0,1]$, too), this is not necessarily the case. For example when we consider a kernel integral operator  $\varrho$ with a continuous kernel, then we may still use $H_1=L^2[0,1]$ and $H_2=C[0,1]$ using the natural embedding of $C[0,1]$ in $L^2[0,1]$.
Alternatively we can resort to the Sobolev space $H_1=W^{1,2}[0,1]$ of once (weakly) differentiable functions in $L^2[0,1]$ equipped with the norm $\|f\|_W = \|f\|_{L^2}  + \|f'\|_{L^2} $; see Chapter 8 of \cite{brezis:2010}. The space $W^{1,2}[0,1]$ is a separable Hilbert space, and because $\|\cdot\|_{\infty} \leq \|\cdot\|_{W}$,
the space $W^{1,2}[0,1]$ can be embedded in $C[0,1]$. This will allow for more general class of continuous operators (for example, including pointwise evaluations).
When viewed with the moment conditions on $\|X_k\|_{H_1}$ implicit to Assumption~\ref{a:hilbertsetting}~(b),
this can be done so long as the covariates $X_k$ are sufficiently smooth.

\section{Some further examples of events $A$}\label{s:partialA}

In addition to level sets and pointwise or uniform prediction sets mentioned in Examples~\ref{E:levelsets} and \ref{example:pred} above, in this section we list some specific examples of sets $A$ that are of interest for the data that we discuss, and which might be useful in other applications.

\begin{Example}[Contrast sets] For some $\gamma\in H_2$ and $a\in\mathbb{R}$ let
\[
A = \left\{ y\in H_2\colon \int_0^1 \gamma(t)y(t)dt>a \right \}.
\]
For example, when $\gamma\equiv 1$, then $A$ is the set of curves which are in average above level $a$.
If $\gamma(t)=2\mathds{1}\{t\leq 1/2\}-1$, or $\gamma(t)=1/2-t$ and $a=E\int_0^1y(t)dt$,
then the set $A$ can be identified as functions with decreasing trend.
\end{Example}

\begin{Example}[Extremal sets.]\label{E:extremalsets}
Let $H_2=C[0,1]$ and let $d\in\mathbb{R}$ and
\[
A = \left\{ y \in H_2\colon \max_{u\in[0,1]} y(u) > d \right\}.
\]
Then $A$ contains functions which will exceed a certain threshold $d$. Note that this is the compliment of a boundary set with bounds $\alpha = -\infty$ and $\beta = d$.
\end{Example}

\begin{Example}[Excursion sets.] Let $H_2=C[0,1]$, $d\in\mathbb{R}$ and $c\in (0,1)$. Set
\[
A = \left\{ y \in H_2\colon \exists \; 0 \leq a < b \leq 1
\text{ with }
b-a \geq c \text{ s.t. } \min_{u\in[a,b]} y(u) > d \right\}.
\]
Then $A$ are the functions which uninterruptedly stay strictly above a certain threshold
for a certain amount of time.
\end{Example}

A crucial condition in  Theorems~\ref{t:main1} and \ref{t:main2}  is that $P(Y\in \partial A)=0$. Below we discuss some examples for which this requirement is fulfilled. For the purpose of illustration, we give details in the case of level sets (Example~\ref{E:levelsets}) and $H_2=C[0,1]$. The other examples can be explored similarly.

\begin{Proposition}\label{l:level1}
Let $\alpha \in \mathbb{R}$ and $z \in [0,1)$.
We define $A=\{y\in C[0,1]\colon \lambda(y>\alpha)\leq z\}$.
The following conditions imply $P(Y\in\partial A)=0$:
\begin{align}
& (i) \; P\big(\lambda(Y=\alpha)>0\big)=0 \quad\text{and}\quad   (ii)\; P\big(\lambda(Y>\alpha)=z\big)=0 && \text{for $z\in (0,1)$},\label{e:condlevel}\\
 &P\big(\sup_{t\in [0,1]} Y(t)=\alpha\big)=0 && \text{for $z=0$}.\label{e:condlevel2}
\end{align}
\end{Proposition}

 The conditions in \eqref{e:condlevel} and \eqref{e:condlevel2} are satisfied by many well known processes, including Brownian motion. They are also generally satisfied by continuously differentiable Gaussian processes under standard non-degeneracy conditions. Such processes might be used to model functional data generated by applying standard smoothing operations, for instance using cubic-splines or trigonometric polynomials, to raw discrete data. We note that comparable differentiability conditions are assumed in \cite{fan:muller:2021}. The following proposition, whose proof we defer to Section~\ref{s:proofs}, describes these conditions.

\begin{Proposition}\label{gauss-lemm}
Suppose that $Y$ is a continuously differentiable Gaussian process with covariance kernel $C_Y$. If $C_Y(t,t)  > 0$ for all $t\in [0,1]$, then \eqref{e:condlevel}(i) holds. For $\ell \in \mathbb{N}$, and $0\le t_1< \cdots  < t_\ell \le 1$, let
\[
r_Y(t,s) = \frac{C_Y(t,s)}{[C_Y(t,t) C_Y(s,s)]^{1/2}}, \mbox{ and } R_{t_1,...,t_\ell} = \{ r_Y(t_i,t_j)\}_{1\le i,j \le \ell} \in \mathbb{R}^{\ell\times \ell}.
\]
If in addition for all $\ell\in \mathbb{N}$ and $0\le t_1< \cdots  < t_\ell \le 1$, there exists constants $c_1, c_2 > 0 $ such that $det( R_{t_1,...,t_\ell}) \ge c_1 \min_{ 1 \le i\ne j \le \ell } | t_i - t_j |^{c_2}$, then  \eqref{e:condlevel}(ii) holds.
If $Y$ is twice continuously differentiable, and
$(Y(t_1),\dots,Y(t_\ell),Y'(t_1),\dots,Y'(t_\ell),Y''(t_1),\dots,Y''(t_\ell))$ has a non-degenerate distribution, then  \eqref{e:condlevel2} holds.
\end{Proposition}

If $A = \{y \in H_2: \langle y, \gamma \rangle > c \}$
is a contrast set with some $\gamma \in H_2$, $c \in \mathbb{R}$,
then from the continuity of the inner product it follows that
$\partial A = \{y \in H_2: \langle y, \gamma \rangle = c \}$, both for $H_2=L^2[0,1]$, and $H_2=C[0,1]$. As for the boundary set
$B_{\alpha,\beta} = \{y\in H_2\colon y([0,1]) \subseteq [\alpha, \beta] \}$, with $H_2=C[0,1]$, $\partial B_{\alpha,\beta} =
\{y\in H_2\colon \sup_{t\in [0,1]} y(t) = \beta \vee \inf_{t\in [0,1]} y(t) = \alpha \}.$

\section{Simulation experiments and data illustrations}\label{s:realdata:sim}

In this section we present the results of several simulation experiments and real data analyses that aimed to evaluate and compare the performance of the algorithms {\bf boot} and {\bf Gauss}, and illustrate their application. We begin by defining some alternate methods that may be used to estimate $P(Y\in A |X)$, and we describe two recent procedures proposed for functional quantile regression and construction of prediction sets in functional data prediction, respectively.

\subsection{Competing methods}

A simple method to estimate $P(Y\in A |X)$ is to employ functional binomial regression.
This entails positing the model
\[
P(Y \in A | X=x) = g\big( \beta_0 + \langle x, \beta \rangle \big)
\]
for some $\beta_0 \in \mathbb{R}$ and $\beta \in L^2[0,1]$, and a link function $g$
that can be chosen from a variety of possibilities, but is most often
the logistic link function, or the cumulative distribution function of a standard normal random variable (the ``probit link''). For more details of such models, we refer to \cite{muller:stadtmuller:2005} and \cite{mousavi:sorensen:2017}. One drawback of note in applying logistic regression in this setting is that changing the set $A$ necessitates refitting the model, which can be computationally cumbersome, and further, as a consequence, the resulting estimators of $P(Y \in A | X)$ need not be monotone with respect to increasing (or decreasing) sets $A$. An approach to adjust such estimators to restore monotonicity is to use rearrangement or isotonization, as discussed in e.g.  \cite{chernozhukov:2010}.

Since the exact relationship between the function $X$ and the event $\{Y \in A\}$
is unknown and difficult to describe in parametric terms, even under model \eqref{mod},
another promising approach is to use nonparametric techniques such a kernel estimators.
Generalizing the method found in Section~5.4 of \cite{ferraty:vieu:2006},
the conditional distribution $P(Y \in A | X = x) = E(\mathds{1}\{Y \in A\} | X = x)$
can be estimated by the functional extension of the
Nadaraya--Watson estimator
\begin{align}
\hat P^\text{NW}(Y\in A|X=x) = \frac{ \sum_{i=1}^n K\big( h^{-1} d(x,X_i) \big)  \; \mathds{1}\{Y_i \in A\} }
                   { \sum_{i=1}^n K\big( h^{-1} d(x,X_i) \big) },
\label{e:nadwat}
\end{align}
where $K$ is a kernel function on the nonnegative real numbers,
$d$ is a distance measure on $H_1$, and
$h > 0$ is a smoothing parameter corresponding to the bandwidth of the kernel.
Note that while the choice of $K$ is typically unproblematic,
the choice of $d$ is more intricate and is often taken to depend on the data.
The bandwidth $h$ represents the trade-off between bias (oversmoothing) and error (undersmoothing),
and is normally taken to decrease with the sample size $n$. \cite{ferraty:vieu:2006} establish consistency conditions for the estimator \eqref{e:nadwat}
in the case when the sequence $\{ (X_k,Y_k)\colon k\geq 1 \}$ is $\alpha$-mixing and $Y$ is scalar. When we apply this method below, we take $K$ to be the standard Gaussian kernel, $d$ to be the norm on $H_1$, and select $h$ using cross-validation.  We note that similarly to functional logistic regression based estimators, a draw back of these estimators is that if one changes the set $A$, then the bandwidth $h$ in general should be recalibrated, and the resulting estimators need not be monotone in $A$ if the bandwidth $h$ is not held fixed for all sets $A$.

Similar options may be derived from the local linear functional estimator,
which improves upon the Nadaraya--Watson estimator by including a linear terms of the form $\langle x - X_i, \beta \rangle$ into the computation of the weights; see \cite{berlinet:2011}. The $k$-nearest neighbors (kNN) functional estimator
is a variation on the Nadaraya--Watson estimator with adaptive bandwidth,
i.e. $h$ is the smallest number such that
$\big|\{ X_i\colon d(x,X_i) \leq h \}\big| = k$.
The kNN estimator has been shown to be consistent for non-parametric regression
by \cite{kudraszow:2013}.

In order to evaluate the proposed algorithms for the construction of prediction sets, we compared to the method of \citet{paparoditis:2020} in the setting of forecasting FAR(1) processes $Y_k-\mu = \varrho (Y_{k-1}-\mu) + \varepsilon_k$. Subsequent to forming the estimator $\hat{\varrho}_n$ using functional principal component analysis, their method entails performing a (sieve) bootstrap on the functional principal component scores of the residuals $\hat{\varepsilon}_{k,n}$ in order to estimate the distribution of the prediction error.
$Y_{n+1}$ is then forecast by $\widehat Y_{n+1} = \hat{\mu}+ \hat{\varrho}_n(Y_n-\hat\mu)$, and uniform prediction sets for $Y_{n+1}$ are constructed of the form
\[
\{ y \in C[0,1] :   \widehat Y_{n+1}(t) + L \, \sigma_{n+1}(t) \le y(t) \le  \widehat Y_{n+1}(t) + U \, \sigma_{n+1}(t), \mbox{ for all } t\in[0,1] \},
\]
where for a specified coverage level $1-\alpha$,
\[
\sigma_{n+1}^2(t) = \widehat\var( \varepsilon^{(n)}(t)), \mbox{ and with }
 M = \sup_{t\in[0,1]} \frac{ |\varepsilon^{(n)}(t) | }{\sigma_{n+1}(t)},
\quad
L = Q_{\alpha/2}(  M),
\mbox{ and }
U = Q_{1-\alpha/2}(M).
\]

In the setting of scalar-on-function quantile regression, we compare to the method of \cite{sang:2020}, which entails for a scalar response $T(Y)$ modelling
\[
T(Y) = g\big( \beta_0 + \langle x, \beta \rangle \big)+ \varepsilon,
\]
where $g$ is an assumed to be unknown link function. The link function $g$ as well as the parameter function $\beta$ are assumed to be linear combinations of splines, and estimated in order to estimate the level $p$ quantile of $T(Y)$ by  minimizing the check function loss
\[
\rho_p(y) = \big(p - \mathds{1}\{y \leq 0\}\big) \, y,
\]
subject also to a roughness penalty on the functions $g$ and $\beta$.

\subsection{Construction of prediction sets}\label{s:pred}
 Following the simulation experiment considered in \cite{paparoditis:2020}, we construct a time series of continuous functions as follows:
\begin{align}
Y_k(t) = \int_0^1 \rho(t,s) Y_{k-1}(s) ds + b\cdot Y_{k-2}(t) + B_k(t), \quad 1\leq k \leq n, \; t\in[0,1],
\label{e:paparoditis}
\end{align}
where $\rho(t,s) = 0.34\, e^{(t^2+s^2)/2}$, and $B_k$ is a standard Brownian motion.
We fit an FAR(1) model to each simulated sample where
we chose the truncation parameter $T_n$ using the PVE criterion with $v=0.85$. This is the same value as used in  \cite{paparoditis:2020}.
If $b$ is chosen as $0$, the FAR(1) model is correctly specified, whereas with $b=0.4$, there is a model misspecification that should be detrimental
to the quality of the model predictions. Following the method proposed in \cite{paparoditis:2020} and as described above, we constructed uniform prediction sets  to forecast each series 1-step ahead, with nominal coverage probabilities of 80\% and 95\%. This was repeated independently 1000 times, with sample sizes $n \in \{ 100,200,400,800\}$. While the model for the forecast is the same for both methods, the difference between our approach and \cite{paparoditis:2020} is in the methods used to estimate the noise distribution $\varepsilon^{(n)}$. These results are summarised in Table~\ref{tab:paparoditis} in terms of empirical coverage probabilities from the 1000 replications.

In the case $n=100$ and $b=0$, the method {\bf boot} yielded empirical coverage probabilities that were up to 4--6 percentage points below the method of \cite{paparoditis:2020}, which are both below the nominal level. Apart from this notable exception, the empirical coverage probabilities are comparable to those of \cite{paparoditis:2020}, and were closer to nominal coverage in 10 out of 16 cases considered. The results of {\bf Gauss} were generally better, which is to be expected since the model errors are Gaussian processes, especially for the nominal coverage probability of 95\%.

\begin{table}
\centering
{\small
\begin{tabular}{llcccccccc}
  \hline
\multicolumn{2}{l}{Nominal} & \multicolumn{2}{c}{$n=100$} & \multicolumn{2}{c}{$n=200$} & \multicolumn{2}{c}{$n=400$} & \multicolumn{2}{c}{$n=800$} \\
\multicolumn{2}{l}{coverage} & $b=0$ & $b=0.4$ & $b=0$ & $b=0.4$ & $b=0$ & $b=0.4$ & $b=0$ & $b=0.4$ \\
  \hline
  80\% & {\bf boot}       &         0.683  &         0.694  &         0.745  &         0.754  &         0.777  &         0.778  &         0.789  &         0.791  \\ %
       & {\bf Gauss}       &         0.703  & \textbf{0.716} &         0.756  & \textbf{0.763} &         0.781  & \textbf{0.783} &         0.791  & \textbf{0.793} \\ %
 & \textit{P., S. (2020)} & \textbf{0.740} &         0.689  & \textbf{0.766} &         0.740  & \textbf{0.791} &         0.768  & \textbf{0.803} &         0.786  \\ %
  \hline
  95\% & {\bf boot}       &         0.861  &         0.872  &         0.913  &         0.917  &         0.933  &         0.936  &         0.944  &         0.944  \\ %
       & {\bf Gauss}       &         0.898  & \textbf{0.904} & \textbf{0.927} & \textbf{0.931} & \textbf{0.940} & \textbf{0.943} & \textbf{0.946} & \textbf{0.946} \\ %
 & \textit{P., S. (2020)} & \textbf{0.902} &         0.856  &         0.918  &         0.899  &         0.927  &         0.913  &         0.936  &         0.924  \\ %
   \hline
\end{tabular}
}
\caption{\label{tab:paparoditis} Empirical coverage probabilities of uniform prediction intervals for the data generating process \eqref{e:paparoditis} calculated via {\bf boot}, as well as using the method of \cite{paparoditis:2020}, abbreviated  \textit{P., S. (2020)}. }
\end{table}

\subsection{Comparison to functional GLM and Nadaraya--Watson estimation}\label{ss:pm10data}

\begin{figure}[t]
\begin{minipage}[b]{0.60\textwidth}
\includegraphics[width=\textwidth]{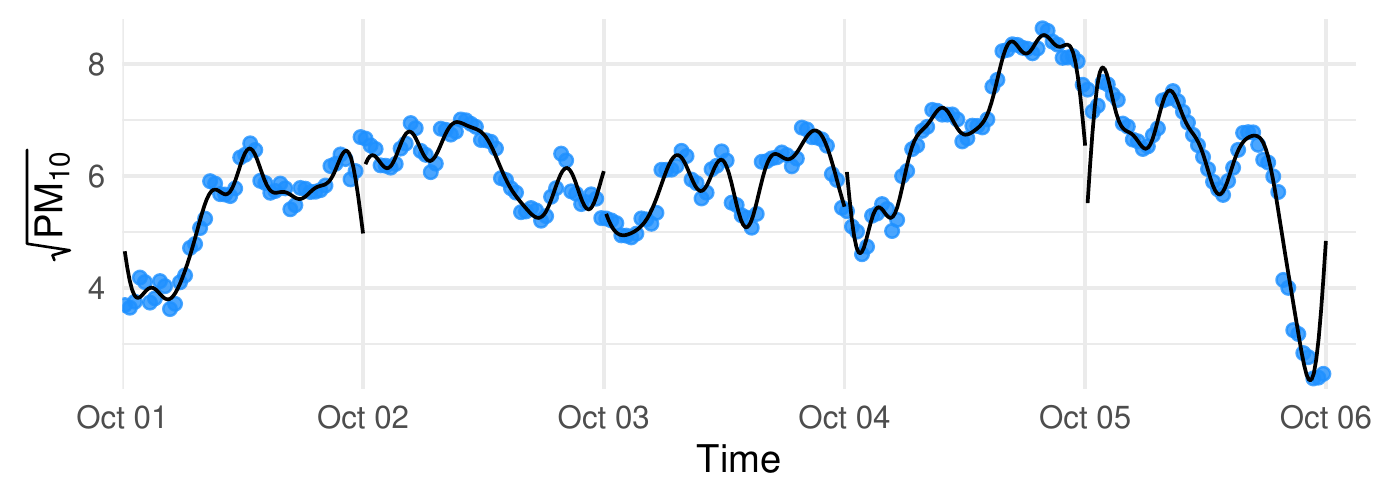}
\includegraphics[width=\textwidth]{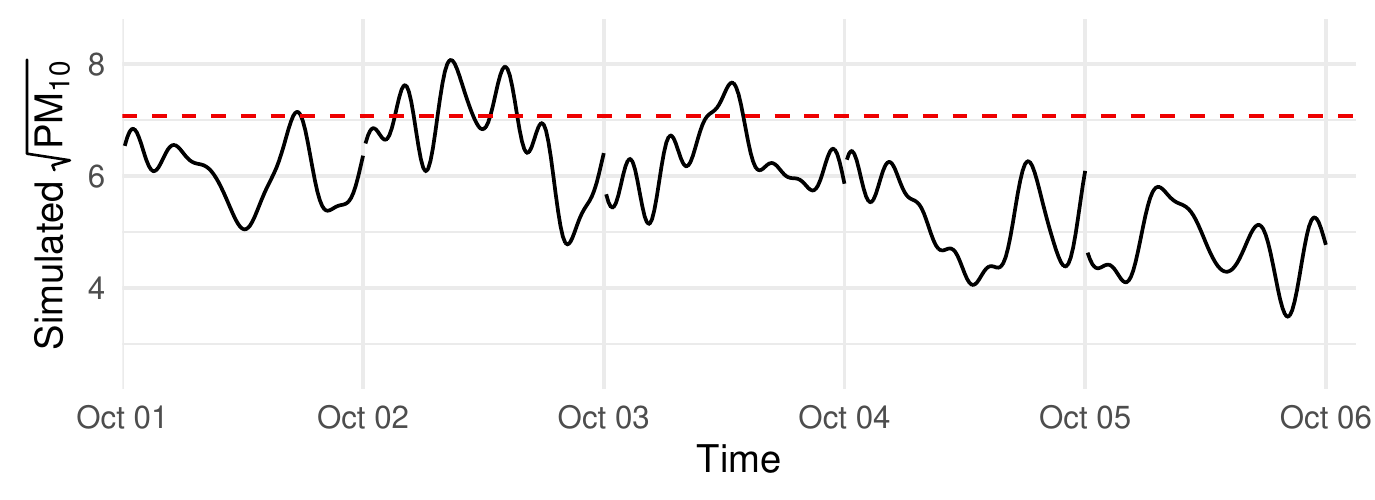}
\end{minipage}
\includegraphics[width=0.39\textwidth]{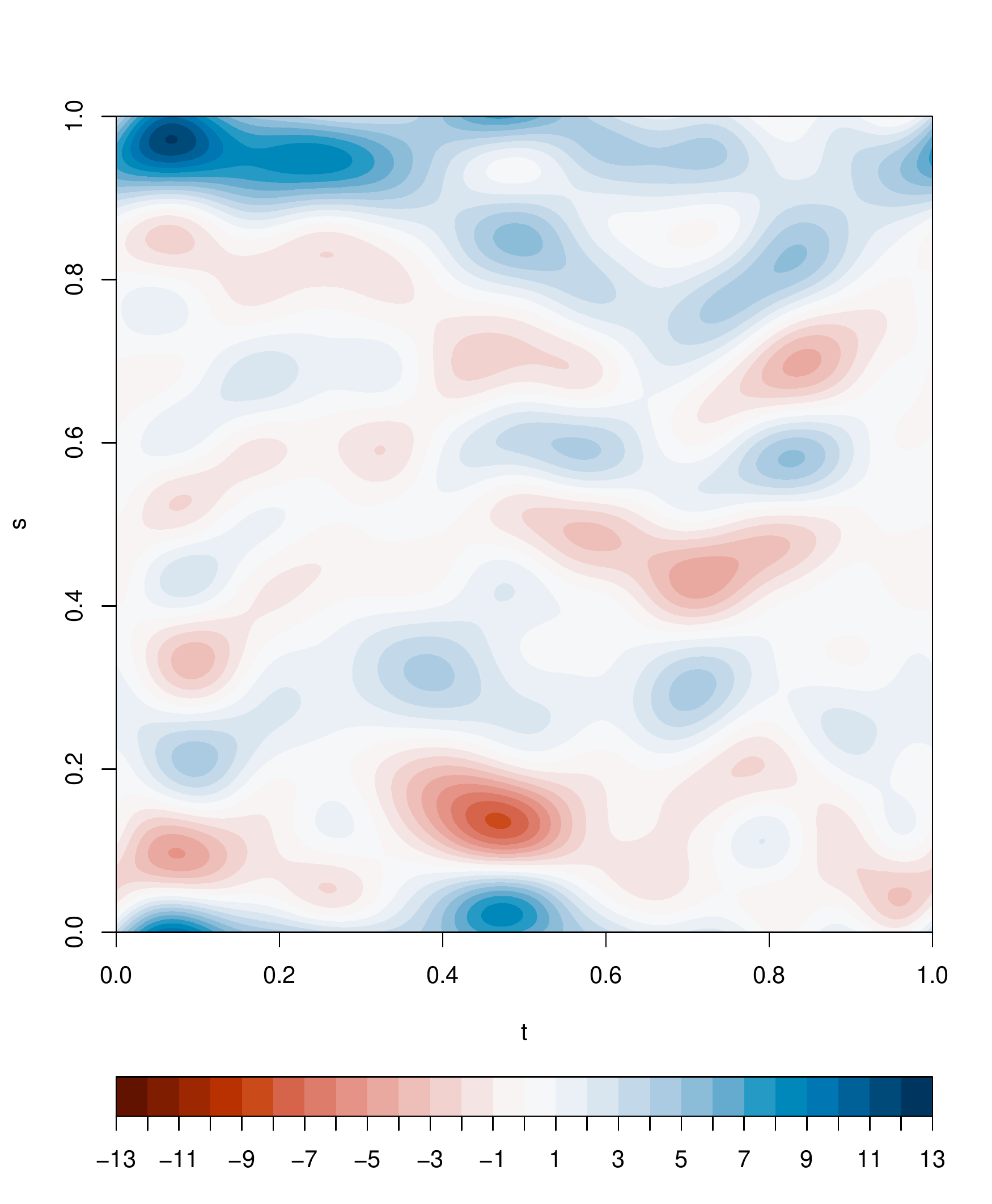}
\caption{Top left: the raw PM$_{10}$ measurements (blue) with the fitted curves (black).
Bottom left: simulated synthetic PM$_{10}$ data (black) with $\alpha = \sqrt{50}$ (red) that we considered in the level set case.
Right: the kernel operator $\varrho(t,s)$ used in the data generating process.}
\label{fig:PM10DGP}
\end{figure}

In this simulation experiment, we generated synthetic data
under model \eqref{mod} in such a way that it resembled a real functional time series derived from daily square-root transformed
PM$_{10}$ concentration curves constructed by smoothing half-hourly measurements of PM$_{10}$.
This is done using the function {\tt Data2fd} in the {\tt fda} package with default settings; see \cite{fdapackage}.
PM$_{10}$ concentration denotes the concentration in air of respirable coarse particles
having a diameter less than 10$\mu m$, and the data that we consider was collected in Graz, Austria over the period from  October 1st, 2010 to March 31st, 2011. An illustration of these data is given in Figure~\ref{fig:PM10DGP}, and they are available in the {\tt ftsa} package in {\tt R}; see \cite{hynd:shang:ftsa:2020}.

We use these data as a means to devise a realistic data generating process.
To this end, we first fit an FAR$(1)$ model to the square-root transformed PM$_{10}$ curves.
The estimator for the FAR operator $\varrho$ obtained in this way differs from operators typically used in simulation settings in that the estimator for the kernel is highly asymmetric, as illustrated in the right hand panel of Figure~\ref{fig:PM10DGP}.

With the estimated sample mean and the fitted FAR operator, we then generate synthetic FAR$(1)$ time series samples by drawing model errors $\varepsilon_k$ from a Gaussian distribution, with the covariance operator estimated from the residuals of the FAR(1) model fit to the original data. This can be done as in the algorithm {\bf Gauss}. The first 30 observations are dropped as part of the burn-in phase.  A snapshot of the raw data in comparison to the synthetic data can be seen in Figure~\ref{fig:PM10DGP}. In this manner we may generate time series of arbitrary sample sizes that are similar to the original PM$_{10}$ data. We generated 1000 independent samples for each sample size $n \in \{50, 100, 250, 1000\}$. Then, for 50 different values of predictors $Y_0^*$, simulated independently from the stationary distribution of the data generating process,  we estimated the conditional probability of $Y^*_{1}$ lying in the level set $P(\lambda(Y_{1}^* > \sqrt{50}) \leq 0.5| Y_0^*)$ for each such sample. For each of the 50 predictors, we also approximated the true probability using Monte-Carlo simulation ($n_{\text{MC}} = 10000$) from the data generating process.

We compared the estimators from algorithms {\bf boot} and {\bf Gauss}, as well as from a logistic functional GLM, and Nadaraya--Watson estimation.
The number $T_n$ of principal components used to estimate $\varrho$ was chosen using  criterion \eqref{e:chooseTn}, so that
\[
T_n = \max\big\{j \geq 1 \colon \hat \lambda_j \geq m_n^{-1} \hat \lambda_1 \big\},
\quad \text{with }\; m_n = 5  n^{0.45}.
\]
We introduce $\hat \lambda_1$ into  the definition of $T_n$ so that the criterion does not depend on the scale of the eigenvalues, yielding a more practicable way of choosing $T_n$.
For $n=1000$, this approximately covers 98\% of the variance of the simulated curves in the sense of the PVE criterion.
Naturally, less variance is covered in for smaller sample sizes.
For the logistic GLM, we used the approach suggested by \cite{muller:stadtmuller:2005}
and took the truncated Karhunen--Loève expansion as the predictor.
In order to keep the methods comparable,
we used the same number $T_n$ of principal components for our algorithms and for the functional GLM.
We calibrated the bandwidth $h$ for the Nadaraya--Watson estimator using leave-one-out cross-validation on each generated sample.
The results in terms of the root mean squared error (RMSE) over the 1000 simulations are displayed in Figure~\ref{fig:PM10RMSElevel}.
Because it is difficult to visualize this for the 50 different predictors, we present boxplots  summarizing the RMSE of each method over all predictors $Y_0^*$.  More details on the results for a variety of specific values of $Y_0^*$ can also be found in Table~\ref{tab:pm10rmselevel} in the Appendix.

\begin{figure}
\includegraphics[width=\textwidth]{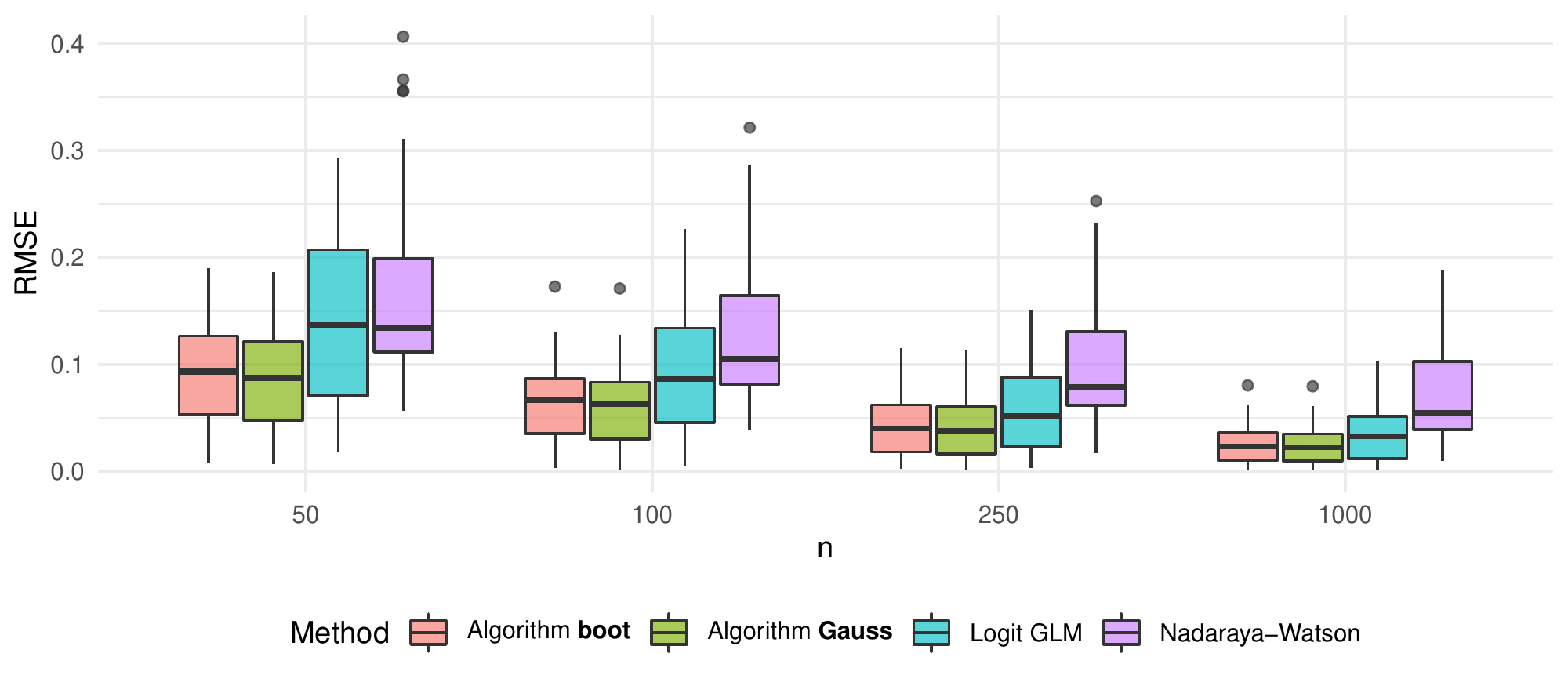}
\caption{ RMSE of $\hat P$ for 50 random predictors $Y_0^*$ and 1000 independent simulations of samples of size $n\in \{ 50, 100, 250, 1000 \}$ based on the estimators {\bf boot}, {\bf Gauss}, functional logistic regression, and Nadaraya--Watson estimators of the probability $P(\lambda(Y_{1}^* > \sqrt{50}) \leq 0.5| Y_0^*)$.}
\label{fig:PM10RMSElevel}
\end{figure}

We observed that algorithms~{\bf boot}  and  {\bf Gauss} exhibited similar predictive performance in both examples and over all  sample sizes. These methods clearly outperformed functional logistic regression and Nadaraya--Watson estimation in estimating the conditional probability of level sets. The proposed methods achieved a similar mean squared error in this case to functional logistic regression with about a quarter of the sample size.
The performance of the Nadaraya--Watson estimator was poor compared to the other methods considered in both cases
and varied strongly depending on the predictor $Y_0^*$. In Appendix~\ref{s:additionalsim}, we also present results in which we considered contrast sets rather than level sets, in which case the same overall pattern was observed, although the results were more comparable across the four methods.

\begin{table}
\centering
\begin{tabular}{|r|rr|rr|rr|rr|}
  \hline
 & \multicolumn{2}{c|}{$n=50$} & \multicolumn{2}{c|}{$n=100$} & \multicolumn{2}{c|}{$n=250$} & \multicolumn{2}{c|}{$n=1000$} \\
 & \multicolumn{2}{c|}{$p=0.98$} & \multicolumn{2}{c|}{$p=0.99$} & \multicolumn{2}{c|}{$p=0.996$} & \multicolumn{2}{c|}{$p=0.999$} \\
$Y_0^*$ & \textbf{boot} & \textbf{Gauss} & \textbf{boot} & \textbf{Gauss} & \textbf{boot} & \textbf{Gauss} & \textbf{boot} & \textbf{Gauss} \\
  \hline
  1 & 0.770 & 0.704 & 0.559 & 0.491 & 0.457 & 0.356 & 0.487 & 0.434 \\
  2 & 0.530 & 0.437 & 0.396 & 0.292 & 0.333 & 0.170 & 0.341 & 0.189 \\
  3 & 0.517 & 0.419 & 0.432 & 0.287 & 0.348 & 0.199 & 0.308 & 0.137 \\
  4 & 0.595 & 0.501 & 0.464 & 0.371 & 0.404 & 0.262 & 0.347 & 0.187 \\
  5 & 0.589 & 0.480 & 0.470 & 0.355 & 0.378 & 0.247 & 0.336 & 0.144 \\
   \hline
\end{tabular}
\caption{\label{tab:pm10varlevel} RMSE for $\hat\alpha_p$, 5 different predictors and 1000 replications. We estimate $\hat\alpha_p$ such that $P(\lambda(Y_{1}^* > \alpha_p)) \leq 0.5 | Y_0^*) = p$, where $p = 1-n^{-1}$.}
\end{table}

Although the estimator {\bf Gauss} performs similarly to {\bf boot} in the above example, it can be expected that {\bf boot} runs into problems when $P(Y \in A|X)$ is close to $0$ or $1$, since {\bf boot} only uses the $n$ estimated model residuals  to estimate $P(Y \in A|X)$, whereas in producing the estimator \textbf{Gauss},
one can generate a Monte-Carlo sample of residuals as large as needed to give a non-degenerate estimate of these  probabilities, which can be expected to be accurate if the Gaussian assumption is plausible.  To highlight this, we present the results of a short simulation study in which for a probability $p_n=1-1/n$, we aimed to estimate $\alpha_p$  using {\bf boot} and {\bf Gauss}  such that $P(\lambda(Y_{1}^* > \alpha_{p_n}) \leq 0.5 | Y_0^*) = p_n$. This problem is hence related to the  Value-at-Risk estimation. We compared the RMSE of $\hat\alpha_{p_n}$ from the two algorithms for 50 different realizations of the predictor $Y_0^*$ that were simulated from the same data generating process. We note that the value of $\alpha_p$ varies between 7.26 and 11.77, depending on $Y_0^*$ and $p_n$. In Table~\ref{tab:pm10varlevel}, we present the results from a subset of five predictors $Y_0^*$ that were representative of the variability observed in the simulated series. It is apparent from these results that {\bf Gauss} outperforms {\bf boot} in all cases, and the relative advantage increases with sample size.
If we look at the results for all 50 predictors, RMSE of $\hat\alpha_p$ decreases by about 15\% for $n=50$, 22\% for $n=100$, 35\% for $n=250$ and 42\% for $n=1000$. This gives some indication of the difference in performance that can be expected between the two methods in forecasting extreme quantiles or events whenever the Gaussian assumption is plausible.

\subsection{Functional quantile regression}\label{ss:quantilereg}

In this application, we compare to the data analysis of \cite{sang:2020}.
As in our previous example, those authors consider the functional time series of daily square-root transformed PM$_{10}$ concentration curves constructed by smoothing half-hourly measurements of PM$_{10}$.

The goal of the analysis is to compare forecasts of the quantiles of the maximum values $M_t = \max_{u\in[0,1]} Y_{t}(u)$ (we remark the relation to Example~\ref{E:extremalsets}), where $Y_t(u)$ is the transformed PM$_{10}$ curve on day $t$ at intraday time $u$. As the covariate the curve $Y_{t-1}$ is used. Now we model the relationship between $(Y_t,Y_{t-1})$ by a FAR(1) process and apply the method {\bf boot} to estimate the conditional quantile of $M_t$.
We select the truncation parameter $T_n$ in order to explain 98\% of the variance in the variables $Y_t$
since for this fixed sample size, tuning $T_n$ by an asymptotic criterion is not meaningful.
At a quantile level $p$, we compared these methods by
5-fold cross-validation the mean
check-function loss $\rho_p\big( M_i - \hat q_p(M_i|Y_{i-1})$.
We did this for seven different quantile levels $p \in \{0.05, 0.15, 0.25, 0.50, 0.75, 0.85, 0.95\}$. The experiment was repeated on 50 random splits of the data set. The results are displayed in Figure~\ref{fig:sangcao}. In 87.4\% of the cases, {\bf boot} outperformed the
functional single-index quantile regression model of \cite{sang:2020} in terms of the loss considered.
This advantage was much smaller for more central quantiles and became more apparent for the more extreme quantiles.

\begin{figure}[t]
\centering
\includegraphics[width=0.90\textwidth]{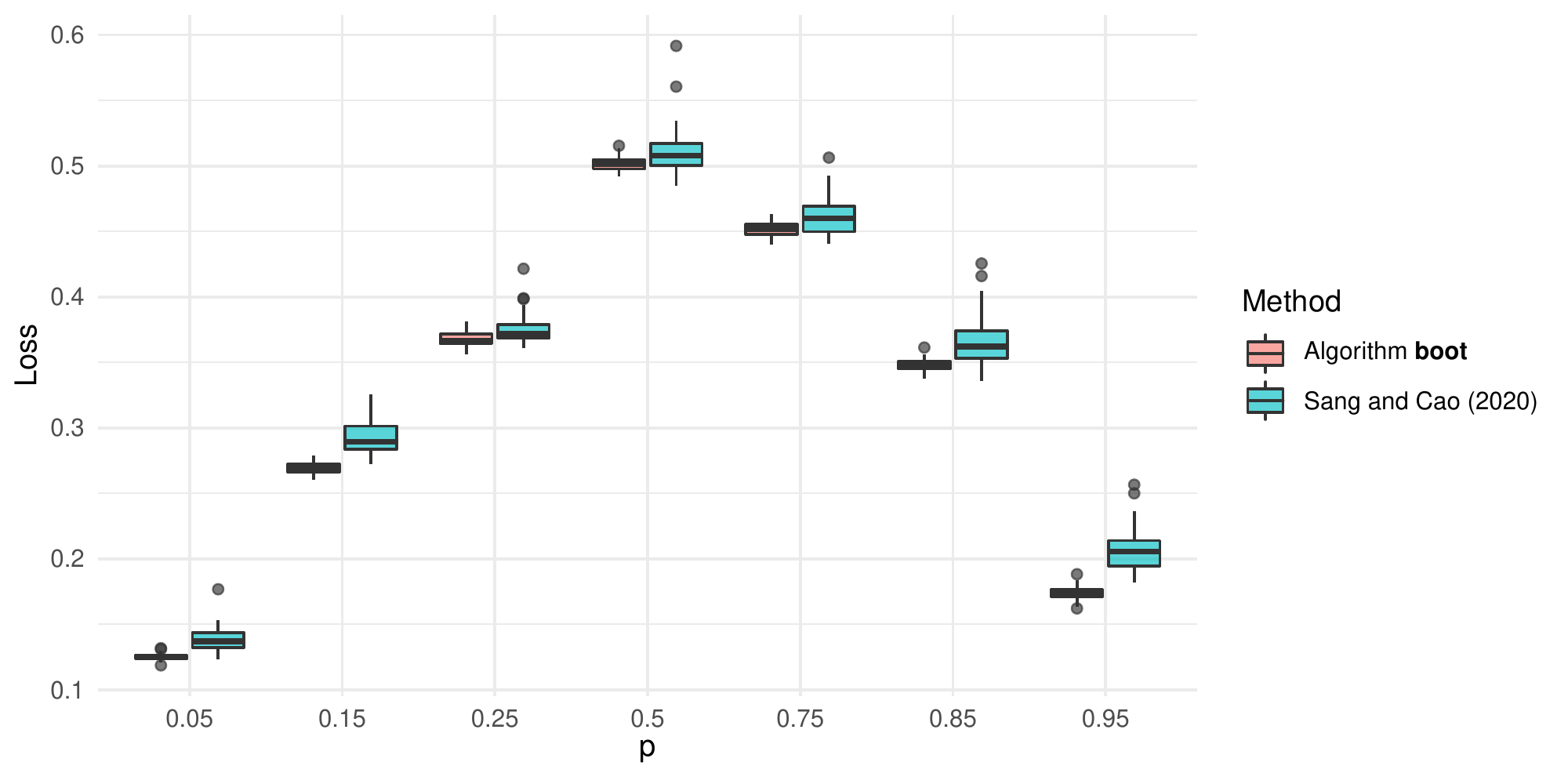}
\caption{Performance of {\bf boot} compared to the
functional single-index quantile regression model proposed by \cite{sang:2020}. The prediction error is compared using 5-fold cross-validation on 50 random splits of the PM$_{10}$ data set.}\label{fig:sangcao}
\end{figure}

\subsection{Spanish electricity price data}\label{ss:electricitydata}

\begin{table}
\centering
\begin{tabular}{lrrrrrrrrrr}
  \hline
      &     $\alpha$    & \hspace{5mm} 30 & 35 & 40 & 45 & 50 & 55 & 60 & 65 & 70 \\
  \hline
      &  \textbf{boot}  & \textbf{0.03} & \textbf{0.06} & \textbf{0.10} & \textbf{0.13} & \textbf{0.16} & \textbf{0.24} &         0.23  & \textbf{0.19} & \textbf{0.16} \\
$z=0$ &  GLM            &         0.11  &         0.17  &         0.23  &         0.36  &         0.30  &         0.24  & \textbf{0.22} &         0.21  &         0.24  \\
      &  N--W           &         0.05  &         0.10  &         0.20  &         0.21  &         0.22  &         0.31  &         0.29  &         0.27  &         0.26  \\
  \hline
      &  \textbf{boot}  & \textbf{0.05} & \textbf{0.07} & \textbf{0.10} & \textbf{0.15} & \textbf{0.15} & \textbf{0.19} & \textbf{0.17} & \textbf{0.16} & \textbf{0.12} \\
$z=\frac 16$ &  GLM     &         0.22  &         0.18  &         0.22  &         0.25  &         0.18  &         0.19  &         0.17  &         0.19  &         0.24  \\
      &  N--W           &         0.08  &         0.14  &         0.18  &         0.22  &         0.25  &         0.26  &         0.23  &         0.22  &         0.18  \\
  \hline
      &  \textbf{boot}  & \textbf{0.05} & \textbf{0.10} & \textbf{0.12} & \textbf{0.13} & \textbf{0.20} & \textbf{0.17} & \textbf{0.15} & \textbf{0.13} & \textbf{0.09} \\
$z= \frac 26$ &  GLM    &         0.25  &         0.15  &         0.26  &         0.23  &         0.23  &         0.17  &         0.23  &         0.20  &         0.39  \\
      &  N--W           &         0.10  &         0.16  &         0.23  &         0.23  &         0.26  &         0.22  &         0.24  &         0.22  &         0.14  \\
  \hline
      &  \textbf{boot}  & \textbf{0.08} & \textbf{0.10} & \textbf{0.11} & \textbf{0.17} & \textbf{0.20} & \textbf{0.17} & \textbf{0.15} & \textbf{0.11} &         0.07  \\
$z= \frac 36$ &  GLM    &         0.11  &         0.15  &         0.26  &         0.27  &         0.23  &         0.19  &         0.18  &         0.20  &         0.37  \\
      &  N--W           &         0.14  &         0.17  &         0.26  &         0.24  &         0.27  &         0.23  &         0.23  &         0.17  & \textbf{0.07} \\
  \hline
      &  \textbf{boot}  & \textbf{0.09} & \textbf{0.11} & \textbf{0.14} & \textbf{0.19} & \textbf{0.22} & \textbf{0.18} & \textbf{0.12} & \textbf{0.08} &         0.03  \\
$z= \frac 46$ &  GLM    &         0.12  &         0.18  &         0.27  &         0.26  &         0.25  &         0.24  &         0.20  &         0.40  &         0.12  \\
      &  N--W           &         0.16  &         0.20  &         0.23  &         0.26  &         0.28  &         0.25  &         0.19  &         0.13  & \textbf{0.02} \\
  \hline
      &  \textbf{boot}  & \textbf{0.14} & \textbf{0.17} & \textbf{0.20} & \textbf{0.24} & \textbf{0.23} & \textbf{0.14} & \textbf{0.08} & \textbf{0.02} & \textbf{0.00} \\
$z= \frac 56$ &  GLM    &         0.21  &         0.26  &         0.25  &         0.26  &         0.27  &         0.17  &         0.21  &         0.13  &         0.10  \\
      &  N--W           &         0.24  &         0.29  &         0.29  &         0.31  &         0.29  &         0.19  &         0.09  &         0.03  &         0.00  \\
   \hline
\end{tabular}
\caption{\label{tab:spaincrossentropy} The cross-entropy of the estimated conditional probability $P(\lambda(Y >\alpha) \leq z)$ for different values $\alpha$ and $z$, evaluated on the test set. The comparison value GLM is a logit regression model with the same predictors. N--W is the Nadaraya--Watson estimator. The smallest value in each cell is marked in bold font, and any apparent ties are merely a result of the rounding to two digits.}
\end{table}

We now return to the Spanish electricity price data that we gave as an introductory example. We take as the goal of this analysis to compare estimates for the conditional probability that the price curves will lie in specified level sets, given the covariates of demand, and wind energy production. In order to compare the various methods for doing this, we split the data into a training and testing set  by randomly taking four months from each year, and assigning them to the test set, which created a 2:1 split between the training and testing set.  Since we used 6 years of this data, the training set thus consists of 1453 days,
and the test set consists of 731 days.
Let $Z_t$ denote one of functional variables electricity price, demand or wind energy production. Then these variables were deseasonalized as follows:
\[
\widetilde Z_t = Z_t - Z_t^{(Y)} - Z_t^{(W)},
\]
where $Z_t^{(Y)}$ is the yearly seasonality obtained
by taking the mean for each day of the year and smoothing the result
using a rolling mean with a window size of 21 days.
$Z_t^{(W)}$ is the weekly seasonality that is estimated as the mean for each day of the week.
For the wind curves, no weekly seasonality was removed. In order to employ the methods {\bf boot} and {\bf Gauss}, we fit the FARX(7) model described in \eqref{e:spainfarx} using the estimator introduced in Section~\ref{s:theoretical} with the data in the training set. The truncation parameter $T_n$ was again chosen in order to explain 98\% of the variance of the covariates.

In order to compare the estimated conditional probabilities to the realized outcomes on the test set, we used the cross-entropy measure.
The cross-entropy of a distribution $P$ relative to a distribution $Q$
is defined as
$H(P,Q) = -\mathbb{E}_P\big[ \log(q(Y)) \big]$,
where $q$ is the probability mass function of $Q$; see Section 2.8 of \cite{murphy2012machine}. Given the realisations $y_i = \mathds{1}\{Y_i \in A\}$, $i \in \{1,...,N\}$, and corresponding estimated conditional probabilities $\hat{p}_i= \hat{P}(Y_i \in A | X_i)$ in the testing set of size $N$, the plug-in estimator of the cross-entropy estimated on the test set is
\[
\hat{H}(\hat{P}) = -\frac 1 N \sum_{k=1}^N [y_k \log\big( \hat{p}_k \big)+ (1-y_k) \log\big(1- \hat{p}_k \big)].
\]

We considered level sets of the form $A= \{ y \in C[0,1] : \lambda(Y > \alpha) \leq z)\}$ for various values of $\alpha$ and $z$, and calculated the cross-entropy on the test set of estimates of $P(Y \in A |X)$ using the method {\bf boot}, as well as for functional logistic regression, which was estimated using the same covariates (and PVE criterion) as those considered in generating the estimator in {\bf boot}, as well as functional Nadaraya--Watson estimation with a Gaussian kernel with the predictors
\emph{demand}, \emph{wind} and \emph{lagged price}, and the bandwidth parameters were selected using leave-one-out cross-validation on the training set. We do not present the results for the method {\bf Gauss}, as the results are again very similar to the method {\bf boot}. The estimated cross entropies on the test set for each set $A$ considered are presented in Table~\ref{tab:spaincrossentropy}. The smallest value in each cell is marked in bold font.

The method {\bf boot} achieved lower values of cross-entropy on the test set compared to the competing methods for most combinations of $\alpha$ and $z$. {\bf boot} had higher estimated cross-entropy in one case compared to functional logistic regression model,
and two cases compared to functional Nadaraya--Watson estimation. The values of $\alpha$ and $z$ considered were chosen such in a way
that most price curves belong to at least one set, and so we do not think this superior performance resulted from the sets $A$ focusing on outcomes that are well modelled using the FARX(7) model.

\section{Summary}

We considered two methods, based on either a residual bootstrap or Gaussian process simulation, to estimate the conditional distribution $P(Y\in A|X)$, where $Y$ and $X$ satisfy the functional linear regression model \eqref{mod}. We showed under mild consistency conditions on the estimated regression operator $\hat{\varrho}_n$ that these methods lead to consistent estimation, in particular in the setting where $Y$ is assumed to be an element of the Banach space $C[0,1]$, which allows for the consideration of sets $A$ that describe more detailed path properties of the response. We put forward one example of an operator estimator $\hat{\varrho}_n$ that has the specified consistency properties under natural regularity conditions on the covariates, which allow for weak serial dependence, and on the choice of tuning parameters. In several simulation experiments and data analyses we observed that these methods generally outperformed prominent competitors, which are often more complicated to implement.

\addcontentsline{toc}{section}{References}
\bibliography{CDbib}
\bibliographystyle{abbrvnat}

\section{Proofs}\label{s:proofs}
\FloatBarrier

Below for a sequence of random variables $X_n$ taking values in a metric space, we use the notation $X_n \stackrel{d}{\to}X$ to denote weak convergence (convergence in distribution). We let $K_*$ denote unimportant positive constants that may change between uses. We state a number of Lemmata, and if their proof is not immediately given, then it is given in Appendix~\ref{a:additionalproofs}.
\begin{lemm}\label{l:1}
Let $(\varepsilon_k)_{k\geq 1}$ be the i.i.d.\ noise sequence of model~\eqref{mod}
and let $(X,Y)$ be drawn independently  from the same regression model.
Then
\begin{equation}\label{e:ET}
\frac{1}{n}\sum_{k=1}^n \mathds{1}\{\varrho(X)+\varepsilon_k\in A\}\stackrel{\mathrm{a.s.}}{\to} P(Y\in A|X),\quad (n\to\infty).
\end{equation}
\end{lemm}

The following slightly modified version of the continuous mapping theorem
can be proven along similar lines as Theorem~1.9.5 of \cite{vaart:wellner:1996}.
\begin{lemm}\label{e:CMTmod}
Let $X_n,Y_n$, $n\geq 1$, and $Y$ be random elements of a metric space. Suppose that $X_n-Y_n\convP 0$ if $n\to\infty$ and that $Y_n\stackrel{d}{=} Y$ for all $n$. Let $g:\mathcal{S}\to\mathbb{R}$ and let $D_g$ be the set of discontinuity points of $g$. Then if $P(Y\in D_g)=0$, we have $g(X_n)-g(Y_n)\convP 0$.
\end{lemm}

\medskip

\noindent\textbf{Proof of Theorem~\ref{t:main1}.} We note that by Lemma~\ref{l:1}, it suffices to show
\begin{equation}\label{e:toshow}
\frac{1}{n}\sum_{k=1}^n \left(\mathds{1}\{\hat\varrho_n(X)+\hat\varepsilon_{k,n}\in A\}-\mathds{1}\{\varrho(X)+\varepsilon_k\in A\}\right)\convP 0.
\end{equation}
To this end we define random variables $K=K_n$ which are uniformly distributed on $\{1,\ldots, n\}$ and independent of $\{(Y_k,X_k)\colon k\geq 1\}$ and $(Y,X)$. A consequence of Assumption~\ref{a:consist}~\ref{a:consistIn} is that
\[
\hat\varepsilon_{K,n}-\varepsilon_K=Y_K-\hat\varrho_n(X_K)-(Y_K-\varrho(X_K))=\varrho(X_K)-\hat\varrho_n(X_K)\convP 0.
\]
Set $V_{K,n}:=\hat\varrho_n(X)+\hat\varepsilon_{K,n}$ and $V_{K,0}:= \varrho(X)+\varepsilon_K.$ Then by Assumption~\ref{a:consist}~\ref{a:consistOut} $V_{K,n}-V_{K,0}\convP 0$.

Consider the indicator function $\mathds{1}\{y\in A\}$ and let $D_A$ be the discontinuity points of $\mathds{1}\{y\in A\}$, i.e.\ $D_A=\partial A$. Noting that $V_{K,0}\stackrel{d}{=} Y$, we deduce by Lemma~\ref{e:CMTmod}
  that  $D_n:=\mathds{1}\{V_{K,n}\in A\}-\mathds{1}\{V_{K,0}\in A\}\convP 0$, provided $P(Y\in \partial A)=0.$ Since $D_n$ is bounded, we get $E|D_n|\to 0$.  Hence
\[
E|D_n| = \frac{1}{n}\sum_{k=1}^n E\left|\mathds{1}\{V_{k,n}\in A\}-\mathds{1}\{V_{k,0}\in A\}\right|\to 0,
\]
and consequently \eqref{e:toshow} holds. \qed

We suppose that $\varepsilon$ is a variable in $H_2$ distributed as a generic model error in \eqref{mod}. 
\begin{lemm}\label{l:noiseconvD}
Suppose that Assumption~\ref{a:consist2} holds. Then $\varepsilon^{(n)} \xrightarrow{d} \varepsilon \mbox{ in $L^2[0,1]$}.$
\end{lemm}

\begin{lemm}\label{l:noiseconvDcont}
Suppose that Assumption~\ref{a:consist2C} holds. Then $\varepsilon^{(n)} \xrightarrow{d} \varepsilon \mbox{ in $C[0,1]$}.$
\end{lemm}

\begin{lemm}\label{l:conditionalslutsky}
Consider a sequence of random variables $\varepsilon^{(n)}\in H_2$ satisfying $\varepsilon^{(n)} \convd \varepsilon$ as $n \to \infty$. Additionally, let $X\in H_1$ be independent from both $(\varepsilon^{(n)})_{n\geq 1}$ and $\varepsilon$. Suppose there exists a mapping $f$ and a sequence of random maps $\hat f_n$, independent from $X$,
such that $f,f_n: H_1 \mapsto H_2$, and
\begin{equation}
\| \hat f_n(X) - f(X) \| \convP 0,
\quad n \to \infty.
\label{a:consistOut2}
\end{equation}
Let $g: H_2 \times H_2 \mapsto H_2$ be a continuous function. Then for all $A$ with $P\big( g\big(\varepsilon, f(X) \big) \in \partial A \big) = 0$,
it holds that
\[
P\big( g\big(\varepsilon^{(n)}, \hat f_n(X) \big) \in A | X \big) \convP P\big( g\big(\varepsilon, f(X) \big) \in A | X \big),
\quad n \to \infty.
\]
\end{lemm}

\medskip

\noindent\textbf{Proof of Theorem~\ref{t:main2}.} Lemmata~\ref{l:noiseconvD} and \ref{l:noiseconvDcont} imply that
that $\varepsilon^{(n)}\convd \varepsilon$ in the metric space $H_2$
if Assumption~\ref{a:consist2} or \ref{a:consist2C} holds.
The  theorem then follows from Lemma~\ref{l:conditionalslutsky}
with $\hat f_n = \hat\varrho_n$ and $g(\eta, W) =   W+ \eta$. \qed

We note that Corollary~\ref{cor-predsets} also follows from Theorems~\ref{t:main1} and \ref{t:main2}, and Lemma~\ref{l:conditionalslutsky}.

\medskip

\noindent\textbf{Proof of Proposition~\ref{p:uniformconv}.} We will adapt the usual Glivenko--Cantelli argument for this proof.
Let $\epsilon, \delta > 0$.
We show first that our assumptions imply that there exists a set of finitely many deterministic points
$a = \xi_0 < \xi_1 < \dots < \xi_K = b$
such that
\[
P\Big[ \big| P(Y \in A_{\xi_k} | X) - P(Y \in A_{\xi_{k-1}} | X) \big| > \delta \Big] < \epsilon/2,
\quad \forall\, 1 \leq k \leq K.
\]
To this end we use the notation $W_\xi = P(Y \in A_{\xi_k} | X)$.
If the interval $[a,b]$ is unbounded,
we can choose $u$ and $v$ such $P(W_{u}< W_{a} + \delta)\geq 1-\varepsilon/2$ and
$P(W_{v} > W_{b} - \delta)\geq 1-\varepsilon/2$.
If the interval $[a,b]$ is bounded, then we do not need this step and set $[u,v] = [a,b]$.
On the compact interval $[u,v]$, the process $W_\xi$ is now uniformly
continuous in probability, i.e.\ for every $\xi$, $\varepsilon$ and $\delta$, there exists a $\gamma$ which doesn't
depend on $\xi$, such that
$
|\xi-\xi'|\leq \gamma
$
implies that $P(|W_\xi-W_{\xi'}|\geq \delta)<\varepsilon/2$. Now set $\xi_i=u+(i-1)\gamma$
and $K=\lfloor (v-u)/\gamma\rfloor+1$.

For any $\xi \in [a,b]$, we have $\xi \in [\xi_{k-1},\xi_k)$ for some $k$ and
with a probability of at least $1-\epsilon/2$,
\begin{align*}
\hat P_n(Y\in A_\xi|X) - P(Y\in A_\xi|X)
 &\leq \hat P_n(Y\in A_{\xi_k}|X) - P(Y\in A_{\xi_{k-1}} |X) \\
 &\leq \hat P_n(Y\in A_{\xi_k}|X) - P(Y\in A_{\xi_k} |X) + \delta.
\end{align*}
Similar arguments yield
\begin{align*}
&\big| \hat P_n(Y\in A_\xi|X) - P(Y\in A_\xi|X) \big| \leq \max_{l \in \{k, k+1\}} \big| \hat P_n(Y\in A_{\xi_l}|X) - P(Y\in A_{\xi_l} |X) \big| + \delta,
\end{align*}
with a probability of at least $1 - \epsilon$. Therefore,
\begin{align*}
&\sup_{\xi \in \mathbb{R}} \Big| \hat P_n(Y\in A_\xi|X) - P(Y\in A_\xi|X) \Big| \leq \max_{0 \leq k \leq K} \big| \hat P_n(Y\in A_{\xi_k}|X) - P(Y\in A_{\xi_k} |X) \big| + \delta.
\end{align*}
The proposition follows
as the maximum on the right hand side converges to 0 in probability. \qed

\bigskip

\noindent\textbf{Proof of Corollary~\ref{cor:unif}.}
By the uniform convergence of $\hat P$ obtained in Proposition~\ref{p:uniformconv} we get
$
\hat P(Y \in A_{\hat\xi_p(X)} | X)-P(Y \in A_{\hat\xi_p(X)} | X)\convP 0.
$
It is clear by definition that
$\hat P(Y \in A_{\hat\xi_p(X)} | X) \geq p$ for all $n$.
We even have $\hat P(Y \in A_{\hat\xi_p} | X) \convP p$,
because of the definition \eqref{e:predictionxi} of $\hat\xi_p(X)$
and the uniform convergence in probability of $\hat P(Y \in A_\xi | X)$
to a continuous function.
This proves the first statement.

As for the second statement, we note that if $P(Y \in A_\xi | X)$ is a.s.\ continuous and strictly increasing in $\xi$,
then (i) $P(Y \in A_{\xi_p(X)} | X)=p$ and %
(ii) for any
$\epsilon>0$  we can find a $\delta_X>0$, such that
$\big|P(Y \in A_{\hat\xi_p(X)} | X) - P(Y \in A_{\xi_p(X)} | X)\big| < \delta_X$ implies
$ |\hat\xi_p(X) - \xi_p(X)| < \epsilon$.
Because $\delta_X > 0$, it follows by (i) and (ii) that
\begin{align*}
P\big( |\hat\xi_p(X) - \xi_p(X)| \geq \epsilon \big)
 &\leq P\Big( \big|P(Y \in A_{\hat\xi_p(X)} | X) - p\big| \geq \delta_X \Big)
 \to 0.
\end{align*}
\qed

\bigskip

\noindent\textbf{Proof of Proposition~\ref{p:main1}.} When $\varrho$ is a bounded linear operator,
$(X_k)$ and $(\varepsilon_k)$
are centered, $L^4$-$m$-approximable random sequences, where
$(\varepsilon_k)$ is white noise independent of
$(X_k)$, then it follows from \cite{hormann:kidzinski:2015}  that
\begin{equation}\label{e:consistentregression}
\big\| \hat \varrho_n(X) - \varrho(X) \big\| \convP 0,
\end{equation}
if $X \stackrel{d}{=} X_1$. In particular \eqref{e:consistentregression}  holds for both, out-of-sample and in-sample prediction,
since if $K_n$ is uniformly distributed on $\{1, 2, \dots, n\}$,
then $X_{K_n} \stackrel{d}{=} X_1$.  \qed

\medskip

In order to lighten the notation,
the covariance operator of a random variable $Z$
will be denoted by  $C_Z=C_{ZZ}$.
The usual operator norm will just be denoted by $\|\cdot\|$.

\bigskip

\noindent\textbf{Proof of  Proposition~\ref{p:main2}.}  We will show that
\[
E \big\| (\varrho - \hat\varrho_n) \widehat C_X (\varrho - \hat\varrho_n)^* \big\|_1
 = \mathcal{O}\Big( \frac{1}{m_n} + \frac{m_n}{n} \Big) \to 0,
\quad n \to \infty.
\]
We will denote the inverse of the truncated estimated covariance as
\[
\widehat C_X^\invtn := \sum_{i=1}^{T_n} \frac{1}{\hat \lambda_i} \, \hat v_i \otimes \hat v_i,
\]
and the projection onto the subspace spanned
by the first $k$ eigenfunctions of $\widehat C_X$
as $\widehat\Pi_k$.
Accordingly, $\widehat\Pi^\perp_k = \mathrm{Id} - \widehat\Pi_k$.
Note that
$\widehat C_X \widehat C_X^\invtn
 = \widehat C_X^\invtn \widehat C_X
 = \widehat \Pi_{T_n}$
and thus
\begin{align}\label{e:varrhodecomp}
\varrho-\hat\varrho_n &= \varrho - \widehat C_{YX} \widehat C_X^\invtn
 = \varrho \widehat \Pi_{T_n}^\perp - \widehat C_{\varepsilon X} \widehat C_X^\invtn.
\end{align}
Then
\begin{align}
(\varrho - \hat\varrho_n) \widehat C_X (\varrho - \hat\varrho_n)^*
 &= \varrho \, \widehat \Pi_{T_n}^\perp \, \widehat C_X \, \varrho^*
    - \widehat C_{\varepsilon X} \widehat \Pi_{T_n}^\perp \, \varrho^*
    - \varrho \, \widehat \Pi_{T_n}^\perp \, \widehat C_{X \varepsilon}
    + \widehat C_{\varepsilon X} \, \widehat C_X^\invtn \, \widehat C_{X \varepsilon}.
\label{e:gammadecomp}
\end{align}
It follows that
\begin{align}
E \big\| (\varrho - \hat\varrho_n) \widehat C_X (\varrho - \hat\varrho_n)^* \big\|_1
 &\leq
    E\big\| \varrho \, \widehat \Pi_{T_n}^\perp \widehat C_X \, \varrho^* \big\|_1
    + 2 E\big\| \varrho \, \widehat \Pi_{T_n}^\perp \widehat C_{X \varepsilon} \big\|_1
    + E\big\| \widehat C_{\varepsilon X} \, \widehat C_X^\invtn \, \widehat C_{X \varepsilon} \big\|_1.
\label{e:gammadecomptrace}
\end{align}
We treat each term separately and show convergence to zero.

Given the definition of $T_n$ in \eqref{e:chooseTn} with $m_n=o\big(\sqrt n\big)$ we get
\begin{align*}
E \big\| \varrho \, \widehat \Pi_{T_n}^\perp \, \widehat C_X \, \varrho^* \big\|_1
 &\leq \|\varrho\|_\cs^2 E \big\| \widehat \Pi_{T_n}^\perp \, \widehat C_X \big\|
 \leq \|\varrho\|_\cs^2 \, E \hat\lambda_{T_n+1}
 < \big\| \varrho \big\|_\cs^2 \, m_n^{-1}.
\end{align*}
By assumption, $\big\| \varrho \big\|_\cs < \infty$ and
 $m_n^{-1} \to 0$.

As for the second term in \eqref{e:gammadecomptrace}, we remark that
$E \big\| \varrho \, \widehat \Pi_{T_n}^\perp \, \widehat C_{X \varepsilon} \big\|_1
 \leq \|\varrho\|_\cs \, E \big\| \widehat C_{X \varepsilon} \big\|_\cs.$
The mean squared Hilbert--Schmidt norm of $\widehat C_{X \varepsilon}$ asymptotically vanishes since
\begin{align*}
n E\| \widehat C_{X \varepsilon} \|_\cs^2
 &= \frac 1 n \sum_{k=1}^n \sum_{l=1}^n E\big[ \langle \varepsilon_k , \varepsilon_l \rangle \langle X_k , X_l \rangle \big]  = \frac 1 n \sum_{k=1}^n E\| \varepsilon_k \|^2 \, E\| X_k \|^2
 = \| \Gamma \|_1 \, \|C_X\|_1 < \infty.
\end{align*}
Note that all summands with $k \ne l$ vanish due to either $\varepsilon_k$ or $\varepsilon_l$ being independent from all other variables. Thus, the second term in \eqref{e:gammadecomptrace} vanishes.

Finally, we note that for the third term in \eqref{e:gammadecomptrace} we have
\begin{align*}
E\big\|\widehat C_{\varepsilon X} \widehat C_X^\invtn \widehat C_{X \varepsilon} \big\|_1
 &\leq E\big[ \| \widehat C_X^\invtn  \| \, \| \widehat C_{X \varepsilon} \|_\cs^2 \big]
 \leq  m_n \, E\| \widehat C_{X \varepsilon} \|_\cs^2  = \frac{m_n}{n}  \, \| \Gamma \|_1 \, \|C_X\|_1.
\end{align*}
Since $m_n = o(\sqrt n)$, this converges to zero,
concluding the proof of the proposition. \qed

\bigskip

\noindent\textbf{Proof of Proposition~\ref{p:mainc1}.} We denote all operator norms that stem from the sup-norm $\|\cdot\|_\infty$ on $C[0,1]$
with $\|\cdot\|_{C}$,
in order to distinguish them from the operator norms with respect to the $L^2$-norm.
If $\varrho\colon H_1 \to C[0,1]$ is bounded, it means that
\[
\|\varrho\|_{C} := \sup_{v \in H_1} \frac{\|\varrho v \|_\infty}{\; \| v \|_{H_1} \!} < \infty.
\]
From the definition, it is clear that by embedding $C[0,1]$ into $L^2[0,1]$,
the inequality
$\|\cdot\|_{L^2} \leq \|\cdot\|_\infty$ implies that
$\|\cdot\|_{L^2} \leq \|\cdot\|_{C}$ for the respective operator norm.

We begin by showing that Assumption~\ref{a:consist} holds.
Let us first remark that in the following proof,
it will not be important whether $X$ is an element of the sample or not; it suffices that $X \stackrel{d}{=} X_1$.
We assume $X_k$ and $\varepsilon_k$ are centered to simplify the presentation. We have according to \eqref{e:varrhodecomp} that
\begin{align*}
\big\|(\hat \varrho_n - \varrho) X \big\|_\infty
 &\leq   \big\| \varrho \widehat\Pi^\perp_{T_n} X \big\|_\infty +  \big\| \widehat C_{\varepsilon X} \widehat C^\invtn_X X \big\|_\infty.
\end{align*}
By Lemmata~10, 12 and 13 in \cite{hormann:kidzinski:2015} the first term is
\begin{equation}\label{e:l13}
\big\| \varrho \widehat\Pi^\perp_{T_n} X \big\|_\infty\leq \|\varrho\|_C \left(\|\Pi^\perp_{T_n}(X) \|_{H_1}+ \|\Pi_{T_n}(X)-\widehat\Pi_{T_n}(X) \|_{H_1}\right)=o_P(1),
\end{equation}
with our choice of $T_n$.

As for the second summand, \eqref{e:chooseTn} implies that
\begin{align*}
\big\| \widehat C_{\varepsilon X} \widehat C^\invtn_X X \big\|_\infty
 &\leq \big\| \widehat C_{\varepsilon X} \big\|_{C} \, \big\| \widehat C^\invtn_X X \big\|_{H_1}
 \leq m_n\big\| \widehat C_{\varepsilon X} \big\|_{C}  \big\| X \big\|_{H_1}.
\end{align*}
For any $v \in H_1$ and fixed $t$,
\begin{align*}
\big(\widehat C_{\varepsilon X} v\big)(t)
 &= \Big\langle \frac 1n \sum_{k=1}^n \varepsilon_k(t) X_k , v \Big\rangle,
\end{align*}
which clearly implies that
\[
\|\widehat C_{\varepsilon X}\|_C
 = \sup_{v \in H_1} \frac{ \sup_{t\in[0,1]} |\big(\widehat C_{\varepsilon X} v\big)(t) | }{\|v\|}
 \leq \sup_{t \in [0,1]} \Big\| \frac 1n \sum_{k=1}^n \varepsilon_k(t) X_k  \Big\|_{H_1}.
\]
Due to independence of $\varepsilon_k$ from $(X_j)_{j\leq k}$,
we have that
\begin{align*}
E\Big\| \frac 1n \sum_{k=1}^n \varepsilon_k(t) X_k  \Big\|^2_{H_1}
 &= \frac{1}{n^2} \sum_{k=1}^n \sum_{l=1}^n E\big[ \varepsilon_k(t) \varepsilon_l(t)  \langle X_k , X_l \rangle \big]\\
 &= \frac{1}{n^2} \sum_{k=1}^n E\big[ \varepsilon_k^2(t) \| X_k \|^2 \big] \leq \frac{1}{n} E\|\varepsilon_k\|_\infty^2 \| C_X \|_1
 = \mathcal{O}(1/n),
\end{align*}
and the right-hand side is independent from $t$.
Let $D \in \mathbb{N}$ and $t_i = i/D$, then
\begin{align*}
\sup_{t\in[0,1]} \Big\| \frac 1n \sum_{k=1}^n \varepsilon_k(t) X_k  \Big\|_{H_1}
 &\leq \max_{i\in \{0,1, \dots, D\} } \Big\| \frac 1n \sum_{k=1}^n \varepsilon_k(t_i) X_k  \Big\|_{H_1}
  + \frac 1n \sum_{k=1}^n M_k D^{-\alpha} \|X_k\|_{H_1}.
\end{align*}
The expectation of the second term is bounded by
$D^{-\alpha} E M_0 \, E\|X_0\|_{H_1}$.
As for the first term,
we first note that for any $t,s\in[0,1]$,
\begin{align}
E \Big\| \frac 1n \sum_{k=1}^n \big( \varepsilon_k(t) - \varepsilon_k(s) \big) X_k\Big\|^2_{H_1}
 &= \frac{1}{n^2} \sum_{k=1}^n E\Big[ \big( \varepsilon_k(t) - \varepsilon_k(s) \big)^2 \|X_k\|_{H_1}^2 \Big] \nonumber \\
 &\leq \frac{1}{n} \, E M_0^2 \, E \|X_0\|_{H_1}^2 \, |t-s|^{2\alpha},
\label{e:cepsxincrement}
\end{align}
due to the independence of $\varepsilon_k$ from $(X_j)_{j\leq k}$.
We can now bound the maximum by the sum of increments and obtain
\begin{align*}
E\Bigg[ \max_{i\in \{0, \dots, D\}} \Big\| \frac 1n \sum_{k=1}^n \varepsilon_k(t_i) X_k  \Big\|_{H_1} \Bigg]
 &\leq E \Big\| \frac 1n \sum_{k=1}^n \varepsilon_k(t_0) X_k  \Big\|_{H_1} \hspace{-4pt} + \sum_{i=1}^D E \Big\| \frac 1n \sum_{k=1}^n \big( \varepsilon_k(t_i) - \varepsilon_k(t_{i-1}) \big) X_k  \Big\|_{H_1} \\
 &= \mathcal{O}\big( n^{-1/2} + D\, n^{-1/2} D^{-\alpha} \big).
\end{align*}
Setting $D=n^{1/2}$, it follows that
\begin{equation}\label{e:p4}
E\Big[ \sup_{t\in[0,1]} \Big\| \frac 1n \sum_{k=1}^n \varepsilon_k(t) X_k  \Big\|_{H_1} \Big] = \mathcal{O}\big( n^{-\alpha/2} \big).
\end{equation}
Therefore, if we choose a sequence
$m_n = o\big( n^{\alpha/2}\big)$
then
$\big\| \widehat C_{\varepsilon X} \widehat C^\invtn_X X \big\|_\infty \convP 0$.
Lastly, we note that these results also hold if we use $X_{K_n}$,
where $K_n$ is drawn independently and uniform on $\{1,\dots,n\}$.
Therefore, Assumption~\ref{a:consist} holds.
This concludes the proof of the proposition. \qed

\clearpage

\appendix

\section{Additional Proofs} \label{a:additionalproofs}

\noindent\textbf{Proof of Lemma~\ref{l:1}.} Let $(\Omega_i,\mathcal{A}_i, \mu_i)$, $i\geq 0$, be probability spaces, such $\Omega_0=H_1$ and $\Omega_i=H_2$ for $i\geq 1$. The $\mathcal{A}_i$ are suitable $\sigma$-algebras on $\Omega_i$. The measure $\mu_0$ characterizes the distribution of $X$ and $\mu_i=\mu_1$ the distribution of $\varepsilon_i$, i.e.\ $P(\varepsilon_i\in A)=\mu_i(A)$.
Then, without loss of generality, we can assume that  $X,\varepsilon_1,\varepsilon_2,\ldots$ are defined on the canonical product space
\[
(\Omega,\mathcal{A},P)=\bigotimes_{i\geq 0} (\Omega_i,\mathcal{A}_i, \mu_i),
\]
with $X(\omega)=w$ and $\varepsilon_i(\omega)=v_i$, where $\omega=(w,v_1,v_2,\ldots)$. It is clear that the transformation $T\colon \Omega\to\Omega$ with $T(\omega)=(w,v_2,v_3,\ldots)$ is $\mathcal{A}\big/\mathcal{A}$-measurable and measure preserving. We can write
\[
\mathds{1}\{\varrho(X)+\varepsilon_k \in A\}=Z\circ T^k,
\]
where $Z(\omega)=\mathds{1}\{\varrho(X(\omega))+\varepsilon_1(\omega)\in A\}$. Hence, by the ergodic theorem we have that
\[
\frac{1}{n}\sum_{k=1}^n \mathds{1}\{\varrho(X)+\varepsilon_k\in A\}\stackrel{\mathrm{a.s.}}{\to} E\left[Z|\mathcal{I}\right],
\]
where $\mathcal{I}$ is the $\sigma$-algebra of the invariant sets of $T$. If $A$ is an invariant set of $T$, it means that $T^{-1}(A)=A$. Hence if $\omega=(w,v_1,v_2,\ldots)\in A$, then $(w,y_1,v_1,v_2\ldots)\in A$ for any $y_1\in H_2$ and by induction we conclude that $(w,y_1,\ldots, y_k,v_1,v_{2}\ldots)\in A$ for any $(y_1,\ldots, y_k)\in H_2^k$ and any $k\geq 1$.
This implies that for any $k\geq 1$ we have
\[
A\in \bigotimes_{i\in \{0,k,k+1,\ldots\}}\mathcal{A}_i=\sigma(X,\varepsilon_k,\varepsilon_{k+1},\ldots).
\]
Thus $A\in \bigcap_{k\geq 1}\sigma(X,\varepsilon_k,\varepsilon_{k+1},\ldots)$.
From \citet[p.~29]{chaumont:2003} %
we deduce that
\[
\mathcal{I}=\bigcap_{k\geq 1}\sigma(X,\varepsilon_k,\varepsilon_{k+1},\ldots)=\sigma(X,\mathcal{T}),
\]
where $\mathcal{T}$ is the tail $\sigma$-algebra $\bigcap_{k\geq 1}\sigma(\varepsilon_k,\varepsilon_{k+1},\ldots)$. Since by Kolmogorov's zero-one law $\mathcal{T}$ contains only events with probability zero or one,
$E\left[Z|\mathcal{I}\right]=E[\mathds{1}\{\varrho(X)+\varepsilon_k \in A\}|X]$. \qed

In the proof of the following lemma, we will use sufficient conditions for tightness
that stem from \cite{suquet:1999}. We restate them for clarity.
\begin{lemm}\label{l:tightness}
Let $\{ \mu_n \}_{n\geq 1}$ be a set of measures on a separable Hilbert space, and
$(v_l)_{l\geq 1}$ an orthonormal system (ONS).
Define $\Pi_k$ the projection operator onto the subspace spanned by $(v_l)_{1 \leq l \leq k}$.
$\{ \mu_n \}_{n\geq 1}$ is tight if and only if
\begin{enumerate}[(a),topsep=1pt,itemsep=0pt]
\item $\forall\, k \geq 1: \quad \{ \mu_n \circ \Pi_k^{-1} \}_{n\geq 1}$ is tight.
\item $\displaystyle
\forall\, \epsilon > 0: \quad
\lim_{k\to\infty} \, \sup_{n} \,
\mu_n\big(\{ x \in H: \| x - \Pi_k x \| > \epsilon \}\big) = 0.
$
\end{enumerate}
\end{lemm}

\medskip

\noindent\textbf{Proof of Lemma~\ref{l:noiseconvD}.} We first show that $\widehat\Gamma_{\varepsilon,n}$ converges to $\Gamma$ in probability.
We use $\widehat C_\varepsilon$  to denote  the empirical covariance operator of the model errors, which is naturally not observable. From Theorem 4.1 of \cite{bosq:2000}, we have that $\widehat C_\varepsilon \convP \Gamma$ as $n \to \infty.$
Let
\begin{align*}
\tilde{\Gamma}_{\varepsilon,n} = \frac{1}{n}\sum_{k=1}^n (\hat\varepsilon_{k,n}-E\hat\varepsilon_{k,n})\otimes(\hat\varepsilon_{k,n}-E\hat\varepsilon_{k,n}).
\end{align*}

It follows from elementary calculations using that  $E\|\varepsilon_k\|^4 < \infty$ that $\| \widehat\Gamma_{\varepsilon,n} - \tilde\Gamma_{\varepsilon,n} \|_1=o_P(1)$, and so in order to simplify the calculations we assume that the $\hat\varepsilon_{k,n}$ terms are properly centered in the definition of  $\widehat\Gamma_{\varepsilon,n}$. For conciseness, we write
\[
\frac 1n \sum_{k=1}^n (\varrho - \hat\varrho_n) X_k \otimes X_k (\varrho - \hat\varrho_n)^*
 = (\varrho - \hat\varrho_n) \widehat C_X (\varrho - \hat\varrho_n)^*.
\]
It holds that
\begin{align*}
\big\| \widehat\Gamma_{\varepsilon,n} - \widehat C_\varepsilon \big\|_1
 &\leq \big\|(\varrho - \hat\varrho_n) \widehat C_X (\varrho - \hat\varrho_n)^*\big\|_1
    + \big\|\widehat C_{\varepsilon X} (\varrho - \hat\varrho_n)^*\big\|_1
    + \big\|(\varrho - \hat\varrho_n) \widehat C_{X \varepsilon}\big\|_1 \\
 &\leq \big\|(\varrho - \hat\varrho_n) \widehat C_X (\varrho - \hat\varrho_n)^*\big\|_1
    + 2 \, \big\|(\varrho - \hat\varrho_n) \widehat C_X (\varrho - \hat\varrho_n)^*\big\|_1^{1/2} \big\| \widehat C_\varepsilon \big\|_1^{1/2},
\end{align*}
where in the last step we used an inequality of the form
$\| R \widehat C_{AB} \|_1 \leq \| R \widehat C_{A} R^* \|_1^{1/2} \|\widehat C_{B}\|_1^{1/2}$.
From Assumption~\ref{a:consist2} it follows via Cauchy--Schwarz
that
$
E \big\| \widehat\Gamma_{\varepsilon,n} - \widehat C_\varepsilon \big\|_1 \to 0
$,
implying that $\widehat\Gamma_{\varepsilon,n} \convP \Gamma$. According to Theorem~2.3 in Bosq,
to show weak convergence it suffices to show that $\langle \varepsilon^{(n)},v\rangle\convd\langle \varepsilon, v\rangle$ for all $v\in H_2$, and that the sequence $(P\circ(\varepsilon^{(n)})^{-1})_{n\geq 1}$ is tight.  To prove this, we note that
\[
\langle \varepsilon^{(n)},v\rangle\stackrel{d}{=} \langle \widehat \Gamma_{\varepsilon,n} v , v \rangle\big.^{1/2}  Z,
\]
where $Z$ is a standard normal random variable which is independent of the sample $((Y_i, X_i)\colon 1\leq i\leq n)$.
By the considerations above it follows that
$\langle \widehat \Gamma_{\varepsilon,n} v , v \rangle\convP \langle \Gamma  v , v \rangle$.
This yields the distributional convergence via Slutzky's lemma.

We will now establish tightness of the measures.
We first fix some ONS $(v_l)_{l\geq 1}$ and note that
$(\Pi_k \varepsilon^{(n)})_{n \geq 1}$ is finite dimensional and

\[
P(\| \Pi_k \varepsilon^{(n)} \| > M) \leq \frac{E\|\Pi_k \varepsilon^{(n)}\|^2}{M^2} = \frac{ \| \Pi_k E \widehat \Gamma_{\varepsilon,n} \|_1 }{M^2} \leq \frac{k \| E \widehat \Gamma_{\varepsilon,n} \| }{M^2}.
\]
As per our assumption,
the numerator is bounded for fixed $k$ and large $n$.
Thus, for each $k$, $(\Pi_k \varepsilon^{(n)})_{n \geq 1}$ is uniformly bounded in probability. Hence the first requirement is met. As for the second requirement,
it is equivalent to
\begin{align}\label{e:tightnessb}
\lim_{k\to\infty} \, \sup_{n} \,
P\big( \sum_{l > k} | \langle \varepsilon^{(n)} , v_l \rangle |^2 > \epsilon \}\big) = 0,\quad
\text{for all $\epsilon > 0$},
\end{align}
where we let $(v_l)_{l \geq 1}$ be the eigenfunctions of $\Gamma$. By Markov's inequality and using $E| \langle \varepsilon^{(n)}  , v_l \rangle |^2=E \langle \widehat \Gamma_{\varepsilon,n} v_l , v_l \rangle$, \eqref{e:tightnessb} follows if
\begin{align*}
 \lim_{k\to\infty} \, \sup_{n} \, \sum_{l > k} E \langle \widehat \Gamma_{\varepsilon,n} v_l , v_l \rangle =0.
\end{align*}
For all self-adjoint and non-negative definite operators,
the trace norm can be written as
$
\| \Gamma \|_1 = \sum_{l\geq 1} \langle \Gamma v_l, v_l \rangle,
$
where all summands are non-negative.
Since
$
\big\| E \widehat\Gamma_{\varepsilon,n} - \Gamma \big\|_1 \leq E \big\| \widehat\Gamma_{\varepsilon,n} - \widehat C_\varepsilon \big\|_1 \to 0,
$
we can choose any $\epsilon > 0$ an
$N\geq 1$ such that for all $n > N$ and for all $k$ it holds that
$\| \Pi_k^\perp E \widehat \Gamma_{\varepsilon,n} \|_1 < \| \Pi_k^\perp \Gamma \|_1 + \epsilon$.
Hence
\begin{align*}
\lim_{k\to\infty} \, \sup_{n} \sum_{l > k} \langle E \widehat \Gamma_{\varepsilon,n} v_l , v_l \rangle
 \leq \max\Big(
 &\max_{1 \leq n \leq N} \lim_{k\to\infty} \sum_{l > k} \langle E \widehat \Gamma_{\varepsilon,n} v_l , v_l \rangle , \lim_{k\to\infty} \sum_{l > k} \langle \Gamma v_l , v_l \rangle + \epsilon
\Big) = \epsilon.
\end{align*}
Since $\epsilon>0$ is arbitrary, this concludes the proof. \qed

\bigskip

\noindent\textbf{Proof of Lemma~\ref{l:noiseconvDcont}.} First we note that under our assumptions,
the estimated noise covariance converges uniformly in probability.
Plugging
$\hat\varepsilon_{k,n} = (\varrho - \hat\varrho_n) X_k + \varepsilon_k$
into the definition of $\widehat \Gamma_{\varepsilon,n}$
and applying the Cauchy--Schwarz inequality along with $(a+b)^2 \leq 2a^2+2b^2$,
we obtain
\begin{align*}
\sup_{t,s\in[0,1]} \Big| \widehat \Gamma_{\varepsilon,n}(t,s) - \Gamma(t,s) \Big|
 &\leq K_* \sup_{t,s} \Bigg| \frac 1 n \sum_{k=1}^n \varepsilon_k\otimes \varepsilon_k(t,s) - \Gamma(t,s) \Bigg| \\
 & \quad + K_* \sup_{t,s} \Bigg| \frac 1 n \sum_{k=1}^n (\varrho X_k- \hat\varrho_n X_k) \otimes (\varrho X_k- \hat\varrho_n X_k)(t,s) \Bigg|.
\end{align*}
The second part converges to zero in probability
by assumption, and for the first part we can apply the law of large numbers in Banach space; see Theorem 2.4 of \cite{bosq:2000}. Therefore, $\widehat \Gamma_{\varepsilon,n}$ converges uniformly in probability to $\Gamma$. To show weak convergence in $C[0,1]$, we apply Theorem~7.5
from \cite{billingsley:1999}.
First, we show that for any finite collection of $t_1, t_2, \dots, t_k \in [0,1]$,
\[
\big( \varepsilon^{(n)}(t_1), \dots, \varepsilon^{(n)}(t_k) \big) \convd \big( \varepsilon(t_1), \dots, \varepsilon(t_k) \big).
\]
The right-hand side of this is multivariate Gaussian with mean zero and
covariance matrix $G = \big( \Gamma(t_i,t_j) \big)_{i,j}$.
Conditionally on $ \widehat\Gamma_{\varepsilon,n}$, the covariance matrix of the left-hand side is $G_n = \big( \widehat\Gamma_{\varepsilon,n}(t_i,t_j) \big)_{i,j}$
and it converges uniformly in probability to $G$.
To show the convergence in distribution of the vector,
the same approach as in the proof of Lemma~\ref{l:noiseconvD} can be used,
and we omit the details.

Next we show the tightness. To this end we prove that
\begin{align}
\lim_{\eta \to 0} \limsup_{n\to\infty}
   P\Big( \sup_{|t-s| \leq \eta} | \varepsilon^{(n)}(t) - \varepsilon^{(n)}(s) | \geq \epsilon \Big) &= 0,
 \quad \forall\, \epsilon>0.
\label{e:tightnesscont}
\end{align}
We will make use of
Dudley's theorem \citep{talagrand:2014}, %
which states that for a Gaussian process $W$ on $[0,1]$,
\[
E\Bigg[ \sup_{d_W(t,s) \leq \eta} | W(t)-W(s) | \Bigg] \leq K \int\limits_0^{\eta} \sqrt{\log N(z,d_W)} dz,
\]
where the pseudometric $d_W$ is defined by $d^2_W(t,s) = \var(W(t)-W(s))$
and $N(z,d_W)$ denotes the number of balls of size $z$ in the pseudometric $d_W$
that are needed to cover the whole interval $[0,1]$.
Thus, if we fix an $n$, then conditionally on $\widehat\Gamma_{\varepsilon,n}$ we have
\[
E\Bigg[ \sup_{V_n(t,s) \leq M_V \eta^\alpha}|\varepsilon^{(n)}(t) - \varepsilon^{(n)}(s)| \, \Bigg| \,  \widehat \Gamma_{\varepsilon,n} \Bigg] \leq K \int\limits_0^{M_V \eta^\alpha} \sqrt{\log N(z,V_n)} dz,
\]
where $K$ is a universal constant.
We also remark that $M_V$ is $\sigma(\widehat\Gamma_{\varepsilon,n})$-measurable.
By Assumption~\ref{a:consist2C}~(b) we have that $V_n^2(t,s) < M_V^2 |t-s|^{2\alpha}$ and hence $|t-s| \leq \eta$ implies $V_n(t,s) < M_V \eta^\alpha$. Moreover, $N(z,V_n) \leq \big( M_V/z \big)^{1/\alpha}$.
Thus, we obtain
\begin{align*}
E \Bigg[ \sup_{|t-s| \leq \eta} |\varepsilon^{(n)}(t) - \varepsilon^{(n)}(s)| \, \Bigg| \,  \widehat \Gamma_{\varepsilon,n} \Bigg]
 \leq K \int_0^{M_V \eta^\alpha} \sqrt{\log N(z,d_n)} dz  \leq \frac{K}{\sqrt \alpha}  M_V \int_0^{\eta^\alpha} \sqrt{\log \big( 1/z \big) } dz.
\end{align*}
Taking the expectations, we see that
\[
E \Bigg[ \sup_{|t-s| \leq \eta} |\varepsilon^{(n)}(t) - \varepsilon^{(n)}(s)| \Bigg] \leq \frac{K}{\sqrt \alpha}  E M_V \int_0^{\eta^\alpha} \sqrt{\log \big( 1/z \big) } dz
 =: R(\eta),
\]
with the bound $R(\eta)$ being independent from $n$ and vanishing as $\eta\to 0$.
In particular we note that
$\varepsilon^{(n)}$ has a modification with continuous sample paths,
i.e. $\varepsilon^{(n)} \in C[0,1]$ for all $n$.
Returning to \eqref{e:tightnesscont},
we obtain by Markov's inequality that
\begin{align*}
\lim_{\eta \to 0} \limsup_{n\to\infty}
   P\Big( \sup_{|t-s| \leq \eta} | \varepsilon^{(n)}(t) - \varepsilon^{(n)}(s) | \geq \epsilon \Big)
 &\leq \frac{1}{\epsilon} \lim_{\eta \to 0} \limsup_{n\to\infty}
   E\Bigg[ \sup_{|t-s| \leq \eta} | \varepsilon^{(n)}(t) - \varepsilon^{(n)}(s) | \Bigg]=0.
\end{align*}
This holds for all $\epsilon$
and therefore implies $\varepsilon^{(n)} \xrightarrow{d} \varepsilon$ in $C[0,1]$. \qed

\bigskip

\noindent\textbf{Proof of Lemma~\ref{l:conditionalslutsky}.} Define the set $C=C_1\cap C_2$ with
\begin{align*}
C_1&=\{x\in H_1\colon \hat f_n(x)\convP f(x)\}
& \!\text{and}\! &&
C_2&=\{x\in H_1\colon P\big( g\big(\varepsilon, f(x)\big) \in\partial A \big)=0\}.
\end{align*}
It can thus be easily deduced that for $x\in C$ we have
$(\hat f_n(x),\varepsilon^{(n)})\convd (f(x),\varepsilon)$. %
This in turn implies by the continuous mapping theorem that $g\big(\varepsilon^{(n)}, \hat f_n(x) \big) \convd g\big(\varepsilon, f(x) \big)$.
For  $A$ such that $P\big( g\big(\varepsilon, f(X) \big) \in \partial A \big) = 0$, and
if $x\in C$, it follows that
\[
h_n(x):=P\big( g\big(\varepsilon^{(n)}, \hat f_n(x) \big) \in A \big)\to h(x):=P\big( g\big(\varepsilon, f(x) \big) \in A \big).
\]
Now notice that if $X$ is independent of $\varepsilon^{(n)}$, $\varepsilon$ and $\hat f_n$, then we have that
\[
h_n(X)=P\big( g\big(\varepsilon^{(n)}, \hat f_n(X) \big) \in A \, | X \big)
\quad\text{and}\quad
h(X)=P\big( g\big(\varepsilon, f(X) \big) \in A \, | X \big).
\]

What is left to show is that $P(X\in C)=1$.
We have by assumption that $P(X\in C_2)=1$.
We know that for all $k \geq 1$,
\begin{align*}
E\Big[ \lim_{n\to\infty} P\big( \| \hat f_n(X) - f(X) \| \geq 1/k \, | X \big) \Big]
 &= \lim_{n\to\infty} E\Big[ P\big( \| \hat f_n(X) - f(X) \| \geq 1/k \, | X \big) \Big] \\
 &= \lim_{n\to\infty} P\big( \| \hat f_n(X) - f(X) \| \geq 1/k \big) = 0,
\end{align*}
where we used dominated convergence %
and \eqref{a:consistOut2}.
Because the expectation vanishes, it must hold that
\[
P\Big( \lim_{n\to\infty} P\big( \| \hat f_n(X) - f(X) \| \geq 1/k \, | X \big) > 0 \Big) = 0,
 \qquad \forall\, k \geq 1.
\]
Therefore, by $\sigma$-subadditivity
\begin{align*}
P(X \in C_1) &=
P\Big( \forall k\colon
        \lim_{n\to\infty} P\big( \| \hat f_n(X) - f(X) \| \geq 1/k \, | X \big) = 0 \Big) = 1.
\end{align*}
Weak convergence follows via the Portmanteau theorem, concluding the proof.\qed

\bigskip

\noindent\textbf{Proof of Proposition~\ref{p:mainc2}.} We first show that Assumption~\ref{a:consist2C}~(a) holds.
We use the same decomposition of $(\varrho - \hat\varrho_n) X_k$
as before to obtain
\begin{align*}
&\sup_{t,s \in[0,1]} \Bigg| \frac 1 n \sum_{k=1}^n (\varrho X_k- \hat\varrho_n X_k)\otimes (\varrho X_k- \hat\varrho_n X_k)(t,s) \Bigg|\\
&\quad \leq K_* \sup_{t,s} \Bigg| \frac 1 n \sum_{k=1}^n \big(\widehat C_{\varepsilon X} \widehat C_X^\invtn X_k\big)\otimes \big(\widehat C_{\varepsilon X} \widehat C_X^\invtn X_k\big)(t,s) \Bigg|\\
 &\qquad+ K_* \sup_{t,s} \Bigg| \frac 1 n \sum_{k=1}^n \big(\varrho \widehat\Pi^\perp_{T_n} X_k\big) \otimes \big(\varrho \widehat\Pi^\perp_{T_n} X_k\big)(t,s) \Bigg|.
\end{align*}
For the first summand on the right we have by \eqref{e:p4}
\begin{align*}
&\sup_{t,s} \Bigg| \frac 1 n \sum_{k=1}^n \big(\widehat C_{\varepsilon X} \widehat C_X^\invtn X_k\big) \otimes \big(\widehat C_{\varepsilon X} \widehat C_X^\invtn X_k\big)(t,s) \Bigg|
 = \sup_{t,s} \Bigg| \frac{1}{n^2} \sum_{i=1}^n \sum_{j=1}^n
  \varepsilon_i(t) \varepsilon_j(s) \langle \widehat C_X^\invtn X_i , X_j \rangle \Bigg| \\
 &\leq m_n \, \sup_{t} \Bigg\| \frac{1}{n} \sum_{i=1}^n \varepsilon_i(t) X_i \Bigg\|_{H_1}^2 = \mathcal{O}_P\big( m_n \, n^{-\alpha/2} \big),
\end{align*}
which by our choice of $m_n$ converges to zero.
For the second summand,
we take the expectation and see that
with our choice of $T_n$,
\begin{align*}
E \Bigg[  \sup_{t,s} \Big| \frac 1 n \sum_{k=1}^n \big(\varrho \widehat\Pi^\perp_{T_n} X_k\big)\otimes \big(\varrho \widehat\Pi^\perp_{T_n} X_k\big)(t,s) \Big|  \Bigg]
 &\leq \|\varrho\|_C^2 \; E \big\| \widehat\Pi^\perp_{T_n} \widehat C_X \widehat\Pi^\perp_{T_n} \big\|
 = \|\varrho\|_C^2 \; E\hat\lambda_{T_n+1} \to 0.
\end{align*}
Thus, the second summand is $o_P(1)$.
Therefore, Assumption~\ref{a:consist2C}~(a) holds. As for Assumption~\ref{a:consist2C}~(b),
we observe
\begin{align*}
\var\big( \varepsilon^{(n)}(t) - {}& \varepsilon^{(n)}(s) \big| \widehat\Gamma_{\varepsilon,n} \big)
 = \frac 1n \sum_{k=1}^n \big( \hat\varepsilon_{k,n}(t) - \hat\varepsilon_{k,n}(s) \big)^2 \\
 &\leq \frac{K_*}{n} \sum_{k=1}^n \big( \varepsilon_k(t) - \varepsilon_k(s) \big)^2
   + \frac{K_*}{n} \sum_{k=1}^n \Big( \big(\varrho \widehat\Pi^\perp_{T_n} X_k\big)(t) - \big(\varrho \widehat\Pi^\perp_{T_n} X_k\big)(s) \Big)^2 \\
 &\quad
   + K_* \, \frac{1}{n} \sum_{k=1}^n \big\langle \widehat C^\invtn_X X_k , \frac 1n \sum_{i=1}^n \big( \varepsilon_i(t) - \varepsilon_i(s) \big) X_i \big\rangle^2.
\end{align*}
Note that
\begin{align*}
\frac{1}{n} \sum_{k=1}^n \big\langle \widehat C^\invtn_X X_k , \,\cdot\, \big\rangle^2
 =
\big\langle \widehat C^\invtn_X \frac{1}{n} \sum_{k=1}^n X_k \otimes X_k C^\invtn_X \,\cdot\, , \,\cdot\, \big\rangle
 =
\big\langle \widehat C^\invtn_X \,\cdot\, , \,\cdot\, \big\rangle
 \leq \|C^\invtn_X\| \, \|\cdot\|^2.
\end{align*}
It follows from Assumption~\ref{a:continuoussetting} that
\begin{align*}
\var\big( \varepsilon^{(n)}(t) - \varepsilon^{(n)}(s) \big| \widehat\Gamma_{\varepsilon,n} \big)
 &\leq \frac{K_*}{n} \sum_{k=1}^n M_k^2  \, |t-s|^{2\alpha}
   + \frac{K_* M_\varrho^2 }{n} \sum_{k=1}^n \| \widehat\Pi^\perp_{T_n}  X_k \|^2  \, |t-s|^{2\alpha} \\
 &\quad
   + K_* \, m_n \, \Big\| \frac 1n \sum_{i=1}^n \big( \varepsilon_i(t) - \varepsilon_i(s) \big) X_i \Big\|^2.
\end{align*}
Taking the expectation and using \eqref{e:cepsxincrement} for the last term,
we have that
\begin{align*}
E\Big[ \var\big( \varepsilon^{(n)}(t) - \varepsilon^{(n)}(s) \big| \widehat\Gamma_{\varepsilon,n} \big) \Big]
 &= \mathcal{O}\big(m_n/n\big) \, |t-s|^{2\alpha}.
\end{align*}
Since $m_n/n \to 0$, the right-hand side is bounded uniformly in $n$
and Assumption~\ref{a:consist2C}~(b) holds.

\medskip

\begin{lemm}\label{l:helplemm1}
Suppose $\|y_n-y\|_\infty\to 0$ and $\lambda(y=\alpha)=0$. Then $\lambda(y_n>\alpha)\to \lambda(y>\alpha)$.
\end{lemm}
\begin{proof} We can consider $y_n$ and $y$ as random variables on the probability space $([0,1],\mathcal{B}([0,1]),\lambda)$. Convergence $\|y_n-y\|_\infty\to 0$ implies also weak convergence~$y_n\convd y$. This
means that $$F_n(t):=\lambda(y_n\leq t)\to F(t):=\lambda(y\leq t)$$
in all continuity points of $F$. The condition $\lambda(y=\alpha)=0$ implies that $\alpha$ is a continuity point.
\end{proof}

\medskip

\noindent\textbf{Proof of Proposition~\ref{l:level1}. }  We define the set $$M=\{y\in C[0,1]\colon \lambda(y>\alpha)=z\}\cup \{y\in C[0,1]\colon \lambda(y=\alpha)>0\}.$$ We show that $\partial A\subset M$ and by assumption $P(Y\in M)=0$. To this end we note that $y\in\partial A$ means that there exists a sequence $(y_n)$ such that (i) $\|y_n-y\|_\infty \to 0$ and one of the pairs
\begin{align*}
\text{(ii)}\quad& y\in A\quad\text{and\quad(iii)}\quad  (y_n)\cap A=\emptyset;\\
\text{(ii')}\quad& y\notin A \quad\text{and\quad(iii')}\quad (y_n)\subset A;
\end{align*}
holds. We assume that (ii) $\lambda(y>\alpha)\leq z$ and  (iii) $\lambda(y_n>\alpha)> z$ for all $n\geq 1$. We show that (i), (ii) and (iii) imply that $y\in M$.  To  achieve a contradiction, we suppose that (i), (ii) and (iii) hold along with
$$
\text{(iv) $\lambda(y>\alpha)\neq z$ and (v) $\lambda(y=\alpha)=0$.}
$$
By (ii) and (iv) it follows that (vi) $\lambda(y>\alpha)< z$.
In Lemma~\ref{l:helplemm1} below we will show that (i) and (v) imply that
$$
\text{(vii) $|\lambda(y>\alpha)-\lambda(y_n>\alpha)|\to 0.$}
$$
Now (vi) and (iii) contradict to (vii). For the case (ii') and (iii') we can argue similarly. This shows part (a). For part (b) we can argue along the same line as for part (a). \qed

\bigskip

\noindent\textbf{Proof of Proposition~\ref{gauss-lemm}.}   Fix $z\in (0,1)$ and $\epsilon > 0$, and let $\lambda(Y>\alpha)=m_\alpha$.  Since $Y^* = Y - \alpha$ is also a continuously differentiable Gaussian process sharing the same covariance as $Y$, we may assume without loss of generality that $\alpha = 0 $ in the definition of $m_\alpha$. Let $N_Y = | \{ t \; : \; Y(t) = 0\}|$, where $|A|$ denotes the number of elements in $A$. By Theorem 1 of \cite{bulinskaya:1961}, $N_Y$ is finite with probability one, which implies \eqref{e:condlevel}(i). Moreover, letting $\ell = \inf\{ |u-v| \; : \; u\ne v,\; \; Y(u) = Y(v) = 0\}$, with $\inf \emptyset = \infty$, it follows from the proof of Theorem 1 in \cite{bulinskaya:1961}, specifically equation (2) therein and the subsequent calculations, that there exists $\ell_\epsilon \in (0,1)$ so that $P( \ell < \ell_\epsilon) < \epsilon$. From this we obtain that

  \begin{align}\label{gauss-calc-1}
  P( m_\alpha = x ) \le  P( \{ m_\alpha = z \} \cap \{ \ell \ge \ell_\epsilon \} )  + \epsilon   = \sum_{k=0}^{\infty} P( \{ m_\alpha = z \} \cap \{ \ell \ge \ell_\epsilon\} \cap \{ N_Y = k \} )  + \epsilon.
  \end{align}
  Clearly $P( \{ m_\alpha = z \} \cap \{ \ell \ge \ell_\epsilon \} \cap \{ N_Y = 0\} ) = 0$. For each $k\ge1$, we define the following collections of points in $(0,1)$, with $t_0=0$ and $t_{k+1}=1$:

 \begin{align*}
  T_{k, \epsilon }^{(1)} = \Biggl\{ 0 < t_1 < \cdots < t_k <1 \; :  \;  \sum_{i \in \{ 1,...,k+1\}, i \mbox{ odd } } t_i - t_{i-1} =z, \; \min_{2 \le i \le k} t_i - t_{i-1} \ge \ell_\epsilon    \Biggl\}.
  \end{align*}

 \begin{align*}
  T_{k, \epsilon }^{(2)} = \Biggl\{ 0 < t_1 < \cdots < t_k <1 \; : \;  \sum_{i \in \{ 2,...,k\}, i \mbox{ even } } t_i - t_{i-1} =z,  \; \min_{2 \le i \le k} t_i - t_{i-1} \ge \ell_\epsilon  \Biggl\}.
 \end{align*}
Set $T_{k, \epsilon }= T_{k, \epsilon }^{(1)} \cup T_{k, \epsilon }^{(2)}$. $T_{k, \epsilon }^{(1)}$ and  $T_{k, \epsilon }^{(2)}$ define the set of all possible points at which the process $Y$ can cross zero $k$ times, with each crossing spaced at least a width $\ell_\epsilon$ apart, such that $m_\alpha = z$, respective of whether the process starts above or below zero. Evidently then for $k\ge 1$,

  \begin{align}\label{gauss-calc-2}
  \{ m_\alpha = z \} \cap \{ \ell \ge \ell_\epsilon\} \cap \{ N_Y = k \}  \subset  \{ \exists (t_1,...,t_k) \in T_{k, \epsilon }, Y(t_1) = \cdots= Y(t_k )  = 0   \} =: \mathcal{A}_k.
  \end{align}
  We define for all $n\ge 1$
  \begin{align*}
  Z_{k,\epsilon}(n)  = \Big\{ 0 < & z_1(n) < \cdots < z_k(n) <1 \; : \; z_i(n) =  \frac{j}{n} \mbox{ for some }  j \in \{1,...,n\},  \\
   & \max_{1\le i \le k } |z_i(n) - t_i| \le \frac{2}{n}, \mbox{ for some } (t_1,...,t_k) \in T_{k, \epsilon } \Big\}.  \\
  \end{align*}
  Elementary calculations making use of the fact that the definitions of $T_{k, \epsilon }^{(1)}$ and  $T_{k, \epsilon }^{(2)}$ prescribe a single linear constraint on the points $t_1,...,t_k\in T_{k, \epsilon }$ give that  $|Z_{k,\epsilon}(n)| \le K_* n^{k-1}$. For $k\ge 1$, let $b_n$ be a sequence of real numbers tending to infinity such that $b_n = o(n^{1/k})$. Then, due to the assumption that $Y$ is continuously differentiable, if $\mathcal{B}_n  = \{ \| Y' \|_\infty > b_n \}$, $P(\mathcal{B}_n)\to 0 $ as $n\to \infty$, where $Y'$ is the derivative of $Y$. It follows that
  $$P(\mathcal{A}_k ) = P(\mathcal{A}_k \cap \mathcal{B}_n^c) + o(1)  \;\; (n\to \infty).  $$
  Note that on the set $\mathcal{B}_n^c$, the derivative of $Y$ is uniformly bounded by $b_n$, and hence on this set at any point $t$ at which $Y(t)=0$, there must exist a nearby point of the form $z(n) = j/n$ for some $j \in \{1,...,n\}$ such that $|t-z(n)|< 2/n$, and $|Y(z(n))| \le K_* b_n/n$. It follows that

  \begin{align*}
    P(\mathcal{A}_k \cap \mathcal{B}_n^c) &\le P \left( \bigcup_{z_1(n),...,z_k(n) \in Z_{k,\epsilon}(n) } \Big\{ \max_{1\le i \le k}|Y(z_i(n))| < \frac{b_n}{n}\Big\} \right)  \\
    &  \le \sum_{z_1(n),...,z_k(n) \in Z_{k,\epsilon}(n) } P\left( \max_{1\le i \le k}|Y(z_i(n))| < \frac{b_n}{n} \right) \le K_* n^{k-1} \frac{b_n^k}{n^k},
  \end{align*}
 where in the last inequality, we used the fact that the joint density of $Y(z_1(n)),...,Y(z_k(n))$ is bounded from above for all $z_1(n),...,z_k(n) \in Z_{k,\epsilon}(n)$. This follows from the assumption that the determinant of the correlation matrix of $Y(z_1(n)),...,Y(z_k(n))$ is bounded from below by a constant multiplied by a positive power of the minimal distance between $z_1(n),...,z_k(n)$, the latter of which is bounded from below for $z_1(n),...,z_k(n) \in Z_{k,\epsilon}(n)$ by $\ell_\epsilon + 2/n$. Letting $n\to \infty$ gives then that $P(\mathcal{A}_k)=0$, which in conjunction with \eqref{gauss-calc-1} and \eqref{gauss-calc-2} implies \eqref{e:condlevel}(ii). When $Y$ is twice continuously differentiable with $(Y(t_1),...,Y(t_d),Y'(t_1),...,Y'(t_d),Y''(t_1),...,Y''(t_d))$ having a non-degenerate distribution, \eqref{e:condlevel2} follows from Theorem 7.3 of \cite{Azais:1991055}.\qed

\FloatBarrier

\setcounter{bottomnumber}{2}
\renewcommand\bottomfraction{0.8}

\clearpage

\section{Additional simulation results }\label{s:additionalsim}
\FloatBarrier

We present here some additional simulation results from Section~\ref{ss:pm10data}.
For each sample we estimated the conditional probability of $Y_{1}^*$ lying in the level set $P(\lambda(Y_{1}^*  > \sqrt{50}) \leq 0.5| Y_0^*)$, and $Y_{1}^* $ lying in the contrast set $P( \langle Y_{1}^*  , \gamma \rangle > 0.5 | Y_0^*)$.
We took the function $\gamma$ to be a step-function that compares the mean between 2-4am and 7-9am, thus estimating the rise of the pollutant levels every morning.
We compared the estimators from algorithms {\bf boot} and {\bf Gauss}, as well as from a probit GLM and Nadaraya--Watson estimation.
We calibrated the bandwidth $h$ for the Nadaraya--Watson estimator using leave-one-out cross-validation on each generated sample.
The results in terms of the root mean squared error (RMSE) over the 1000 simulations are displayed in Figure~\ref{fig:PM10RMSEcontrast} for the contrast set.
Additional results for specific values of $Y_0^*$ are displayed in Table~\ref{tab:pm10rmsecontrast}.
We remark that the choice of a probit or a logit link in the GLM
appeared to have little influence on the predictive performance of the level set case.
For contrast sets, given how the data was generated,
the functional logistic regression model with probit link is in fact correctly specified,
and we saw in this case that it tended to perform similarly well as {\bf boot} and {\bf Gauss} in large sample sizes.
In smaller sample sizes, its performance suffers from the lost information during the estimation.

\begin{table}
\centering
{\small
\begin{tabular}{|rr|rrrr|rrrr|}
  \hline
 &   & \multicolumn{4}{c|}{$n=50$} & \multicolumn{4}{c|}{$n=100$} \\
 $Y_0^*$ & ${P(Y_{1}^*  \in A|Y_0^*)}$\! & \textbf{boot} & \textbf{Gauss} & GLM & N--W & \textbf{boot} & \textbf{Gauss} & GLM & N--W \\
  1 & 0.053  & 0.139 & 0.134 & 0.190 & 0.356  & 0.086 & 0.083 & 0.099 & 0.252 \\
  2 & 0.449  & 0.157 & 0.155 & 0.238 & 0.217  & 0.126 & 0.122 & 0.165 & 0.166 \\
  3 & 0.772  & 0.095 & 0.088 & 0.139 & 0.140  & 0.068 & 0.063 & 0.100 & 0.112 \\
  4 & 0.914  & 0.066 & 0.059 & 0.092 & 0.189  & 0.045 & 0.041 & 0.062 & 0.160 \\
  5 & 0.969  & 0.037 & 0.033 & 0.045 & 0.134  & 0.022 & 0.018 & 0.027 & 0.101 \\
  \hline
 &   & \multicolumn{4}{c|}{$n=250$} & \multicolumn{4}{c|}{$n=1000$} \\
 $Y_0^*$ & ${P(Y_{1}^*  \in A|Y_0^*)}$\! & \textbf{boot} & \textbf{Gauss} & GLM & N--W & \textbf{boot} & \textbf{Gauss} & GLM & N--W \\
  1 & 0.053  & 0.052 & 0.050 & 0.054 & 0.181  & 0.034 & 0.033 & 0.036 & 0.129 \\
  2 & 0.449  & 0.080 & 0.077 & 0.104 & 0.129  & 0.045 & 0.045 & 0.060 & 0.101 \\
  3 & 0.772  & 0.046 & 0.042 & 0.066 & 0.080  & 0.028 & 0.028 & 0.048 & 0.052 \\
  4 & 0.914  & 0.028 & 0.027 & 0.039 & 0.126  & 0.019 & 0.018 & 0.023 & 0.089 \\
  5 & 0.969  & 0.014 & 0.012 & 0.017 & 0.075  & 0.010 & 0.009 & 0.009 & 0.048 \\
  \hline
\end{tabular}
}
\caption{ \label{tab:pm10rmselevel} RMSE for different sample sizes, 5 different predictors and 1000 replications for simulated PM$_{10}$ data. The comparison value GLM is a logistic regression model, N--W is the Nadaraya--Watson estimator. We estimate $P(\lambda(Y_{1}^* > \sqrt{50}) \leq 0.5 |Y_0^*)$. }
\end{table}

\begin{figure}[t]
\includegraphics[width=\textwidth]{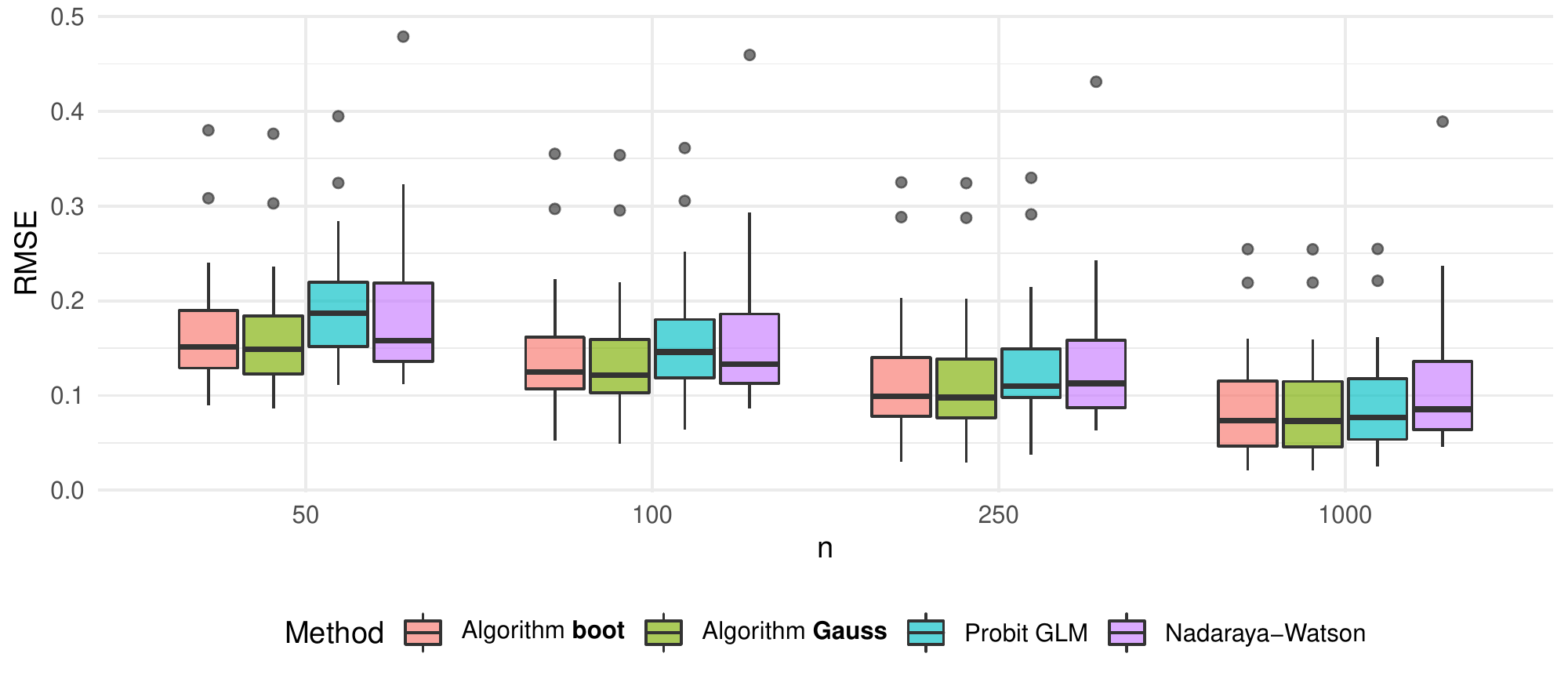}
\caption{ RMSE of $\hat P$ for 50 random predictors $Y_0^*$ and 1000 independent simulations of samples of size $n\in \{ 50, 100, 250, 1000 \}$ based on the estimators {\bf boot}, {\bf Gauss}, functional logistic regression, and Nadaraya--Watson estimators of the probability that the response lies in the contrast set $P( \langle Y_{1}^* , \gamma \rangle > 0.5 | Y_0^*)$.}\label{fig:PM10RMSEcontrast}
\end{figure}

\begin{table}
\centering
{\small
\begin{tabular}{|rr|rrrr|rrrr|}
  \hline
 &   & \multicolumn{4}{c|}{$n=50$} & \multicolumn{4}{c|}{$n=100$} \\
 $Y_n^*$ & ${P(Y_{1}  \in A|Y_n^*)}$\! & \textbf{boot} & \textbf{Gauss} & GLM & N--W & \textbf{boot} & \textbf{Gauss} & GLM & N--W \\
  1 & 0.161  & 0.090 & 0.086 & 0.111 & 0.166  & 0.067 & 0.065 & 0.079 & 0.132 \\
  2 & 0.110  & 0.149 & 0.147 & 0.155 & 0.232  & 0.137 & 0.135 & 0.138 & 0.219 \\
  3 & 0.524  & 0.180 & 0.177 & 0.200 & 0.179  & 0.163 & 0.160 & 0.172 & 0.167 \\
  4 & 0.417  & 0.149 & 0.144 & 0.187 & 0.123  & 0.115 & 0.112 & 0.145 & 0.102 \\
  5 & 0.356  & 0.194 & 0.192 & 0.227 & 0.150  & 0.170 & 0.169 & 0.192 & 0.127 \\
  \hline
 &   & \multicolumn{4}{c|}{$n=250$} & \multicolumn{4}{c|}{$n=1000$} \\
 $Y_n^*$ & ${P(Y_{1}  \in A|Y_n^*)}$\! & \textbf{boot} & \textbf{Gauss} & GLM & N--W & \textbf{boot} & \textbf{Gauss} & GLM & N--W \\
  1 & 0.161  & 0.057 & 0.056 & 0.062 & 0.085  & 0.049 & 0.048 & 0.051 & 0.049 \\
  2 & 0.110  & 0.120 & 0.119 & 0.123 & 0.192  & 0.086 & 0.086 & 0.087 & 0.160 \\
  3 & 0.524  & 0.150 & 0.149 & 0.154 & 0.149  & 0.108 & 0.107 & 0.111 & 0.140 \\
  4 & 0.417  & 0.082 & 0.080 & 0.099 & 0.070  & 0.049 & 0.048 & 0.056 & 0.050 \\
  5 & 0.356  & 0.152 & 0.151 & 0.162 & 0.114  & 0.118 & 0.117 & 0.121 & 0.110 \\
  \hline
\end{tabular}
}
\caption{\label{tab:pm10rmsecontrast} RMSE for different sample sizes, 5 different predictors and 1000 replications for simulated PM$_{10}$ data. The comparison value GLM is a probit regression model, N--W is the Nadaraya--Watson estimator. We estimate the probability of the contrast set $P(\langle Y_{1}^* , \gamma \rangle > 0.5 | Y_0^*)$.}
\end{table}

\end{document}